\documentclass[11pt,reqno]{amsart}

\usepackage{amsmath,amssymb,amsthm,mathtools,bm}
\usepackage{enumitem}
\usepackage{xcolor}
\usepackage{microtype}
\usepackage{comment}
\usepackage[
    colorlinks=true,
    linkcolor=blue!55!black,
    citecolor=blue!55!black,
    urlcolor=blue!55!black
]{hyperref}

\allowdisplaybreaks
\numberwithin{equation}{section}

\theoremstyle{plain}
\newtheorem{theorem}{Theorem}[section]
\newtheorem{lemma}[theorem]{Lemma}
\newtheorem{proposition}[theorem]{Proposition}
\newtheorem{corollary}[theorem]{Corollary}

\theoremstyle{definition}

\theoremstyle{remark}

\newcommand{\R}{\mathbb R}
\newcommand{\tr}{\operatorname{tr}}
\newcommand{\Div}{\operatorname{div}}

\newcommand{\ip}[2]{\left\langle #1,#2\right\rangle_{D_k}}
\newcommand{\dnorm}[1]{\lvert #1\rvert_{D_k}}

\title[Interior Curvature Estimates for Scalar Curvature Equations]
{Interior Curvature Estimates of Semi-convex Solutions for the scalar curvature equation  with Lipschitz Right-Hand Sides}

\author{Lichun Liang}
\address{
School of Mathematical Sciences, Chongqing Normal University, Chongqing, 401331, PR China
}
\email{lianglichun@cqnu.edu.cn}

\thanks{Corresponding author: Lichun Liang.}



\date{}

\begin{document}

\begin{abstract}
In this paper, let $u\in C^4(B_{10})$ with $D^2u\geq -KI$ define a $2$-admissible graph $M=\{(x,u(x)):x\in B_{10}\}\subset\R^{n+1}$ satisfying
\[ \sigma_2(\kappa[u])=f(x).
\]
We prove an interior curvature estimate depending on the Lipschitz norm of the right-hand sides.  The proof combines a shifted Jacobi inequality for
\(b=\log(H+J_0)\), a parallel hypersurface transformation that makes
the Newton tensor uniformly elliptic and local boundedness estimate then reduces the
pointwise bound to a weighted $L^1$ estimate, which is completed using the
Jacobi energy inequality and integration by parts.
\end{abstract}
\maketitle



\section{Introduction}
\noindent

For a graph $M=\{(x,u(x)):x\in B_{10}\}\subset\R^{n+1}$, we denote its
principal curvatures by $\kappa[u]=(\kappa_1,\ldots,\kappa_n)$.  We study
interior curvature estimates for the prescribed scalar curvature equation
\begin{equation}\label{eq:intro-main}
 \sigma_2(\kappa[u])=f(x)>0,
 \qquad \kappa[u]\in \Gamma_2,
\end{equation}
where \[
 \Gamma_2\Psi=\{\lambda\in\R^n:\sigma_1(\lambda)>0,
                         \ \sigma_2(\lambda)>0\}.
\]

Our main result is the following.

\begin{theorem}\label{thm:main}
Let $n\geq 2$, $K\geq 0$, $u\in C^4(B_{10})$ and $f\in C^2(B_{10})$ with $D^2u\geq -KI$ and $\inf_{B_{10}}f >0$. If $M=\{(x,u(x)):x\in B_{10}\}\subset\R^{n+1}$ is a smooth 2-convex graph satisfying the scalar curvature equation \eqref{eq:intro-main}, then we have 
$$\sup_{B_{1/2}}|\kappa(x)|\leq C(n, K, \inf_{B_{10}}f, \|f\|_{C^{0,1}(B_{10})}, \|Du\|_{L^{\infty}(B_{10})})$$
\end{theorem}

The problem originates in the  curvature estimates used in
the Weyl and Minkowski problems, see \cite{Nir53,Pog73,Pog78}.  Heinz's \cite{Heinz} two-dimensional estimate revealed some special properties of solutions to Monge--Amp\`ere type equations.   Unfortunately, Pogorelov \cite{Pog78}  constructed
singular convex solutions of the Monge--Amp\`ere equation in dimensions
$n\ge3$ and Urbas \cite{Urbas} produced analogous counterexamples for
the $k$-Hessian equation for $k\geq 3$.  Consequently, interior $C^2$ estimates for the $2$-Hessian equation become an important problem worthy of further investigation.

For the $2$-Hessian equation, the first  interior $C^2$ estimate is
the two-dimensional result of Heinz.  Warren-Yuan \cite{WarrenYuan} used the special Lagrangian structure  to obtain the interior $C^2$ estimate for $\sigma_2(D^2u)=1$ in dimension three.  Qiu \cite{QiuHessian,QiuCurvature} subsequently treated $C^{1,1}$ right-hand sides in dimension three for both the $2$-Hessian equation and the prescribed scalar curvature equation for graphs. In dimension four, Shankar-Yuan \cite{ShankarYuanFour} proved the interior $C^2$ estimate for $\sigma_2(D^2u)=1$ by combining an almost Jacobi inequality with a doubling argument; their method also gives a new proof in dimension three. Fan \cite{Fan} extended the four-dimensional result to
$C^{1,1}$ right-hand sides. In high dimensions $n\geq 5$, however, all known interior $C^2$ estimates require some  convexity assumption.  McGonagle-Song-Yuan \cite{McGonagleSongYuan} obtained interior $C^2$ estimates for almost convex solutions of  $\sigma_2(D^2u)=1$ by compactness.  Shankar-Yuan \cite{ShankarYuanSemiconvex} providing an integral approach to interior $C^2$ estimates for 
semi-convex solutions of $\sigma_2(D^2u)=1$. Also, under a dynamic semi-convexity condition, Shankar-Yuan \cite{ShankarYuanFour} obtained the interior $C^2$ estimate for  $\sigma_2(D^2u)=1$. For $C^{1,1}$ right-hand sides, the interior $C^2$ estimate was established by Guan-Qiu \cite{GuanQiu}  under the structural condition
$\sigma_3(D^2u)\ge-A$. Moreover, Mooney \cite{Mooney} established strict
$2$-convexity for convex solutions of  $\sigma_2(D^2u)=1$, thereby establishing the regularity for convex viscosity solutions. After sustained progress in weakening the regularity assumptions on the right-hand sides, \emph{a priori} estimates and regularity for  2-Hessian equations has now reached the classical Schauder theory. In a pioneering work, Zhou \cite{Zhou} established the regularity and interior $C^2$ estimates for viscosity solutions of $\sigma_2(D^2u)=f$ with $f\in C^{0,1}$ in dimension three by exploiting its connection with the twisted special Lagrangian equation. Subsequently, Chen-Jian-Zhou \cite{ChenJianZhou} established interior $C^2$ estimates in all dimensions for convex solutions of $\sigma_2(D^2u)=f$ with Lipschitz right-hand side $f$. Following this work, Chen-Jian-Tu-Zhou \cite{ChenJianTuZhou} established interior regularity for viscosity solutions of $\sigma_2(D^2u)=f$ with Lipschitz right-hand side $f$.
Li-Wu \cite{LiWuQuadratic} established interior $C^2$ estimates and regularity for 2-convex solutions of $\sigma_2(D^2u)=f$ with $f\in C^{0,1}$ in all dimensions $n\geq2$. Zhou-Zhu \cite{ZZ26} obtained the Schauder estimates and regularity for convex solutions of 2-Hessian equations. Chen-Zhou-Zhu \cite{CZZ26} completed the Schauder theory for \(2\)-convex solutions of the quadratic Hessian equation.

The geometric equation is technically more delicate for the reason that  the covariant second-derivative calculations generate curvature commutator terms that have no counterpart in the Hessian setting. Guan-Qiu \cite{GuanQiu} established curvature estimates for \eqref{eq:intro-main} in arbitrary dimensions under the condition \(\sigma_3(\kappa)\geq -A\), assuming that the prescribed curvature \(f\) belongs to \(C^2\). In dimension three, Qiu \cite{QiuCurvature} removed the convexity assumption and established an interior curvature estimate for  solutions of
\(\sigma_2(\kappa)=f(X,\nu)>0\) through an analogue of the special Lagrangian integral structure.  Chen–Jian–Zhou \cite{ChenJianZhou} obtained the curvature estimates for convex solutions with constants depending  on the Lipschitz norm of $f$.
Fan-Shankar \cite{FanShankar} treated $\sigma_2(\kappa)=1$ in dimension four by adapting the doubling method.  However, their bound is implicit through
the modulus of continuity of $Du$; for $n\geq5$ their argument requires the
dynamic condition $\kappa_{\min}/H\geq-c(n)$.  Finally, Qiu and Yan \cite{QiuYan} established curvature estimates for   admissible solutions to the constant graphical scalar curvature equation  in all dimensions $n\geq 3$.

These developments naturally raise the following  question.

\medskip
\noindent\textbf{Question.}
For the scalar curvature equation\[
\sigma_2(\kappa[u])=f(x)>0,\]
what is the minimal regularity required of the positive right-hand side \(f\) for an interior \emph{a priori} curvature estimate?

In this paper, our main aim is to giving an answer for the semi-convex solutions of scalar curvature equations with Lipschitz right-hand sides.  

We conclude the introduction with a brief overview of the main ideas underlying the proof of Theorem \ref{thm:main}. We follow a prevailing strategy developed in \cite{ChenJianZhou,ShankarYuanSemiconvex}, which comprises three key ingredients. Firstly, we establish a strong shifted Jacobi inequality for $\log (H+a)$ and reformulate it in divergence form. Secondly, using the parallel hypersurface transform together with a mean-value-type inequality, we control the pointwise value $\log (H(0)+a)$ by an appropriate integral quantity. Finally, we estimate the resulting integral terms through integration by parts. This idea has already been realized in the work of Ji-Liang \cite{JiLiang26}. 

The term $\Delta_M f$ appears in the trace--Jacobi identity for $H$ and is treated weakly in divergence form.  The curvature-commutator term, however, is bounded below only by \(-CH^2\) and the resulting quadratic error cannot be controlled in the subsequent integral argument.  Passing to the logarithmic Jacobi quantity divides the error by $H$ and reduces it to first order $H$, which will appear as the uniformly bounded term $H/J$
via parallel hypersurface transformation.

The negative gradient term generated by $\log H$ requires the third-order quadratic form to provide a coefficient strictly greater than $1$, which fails  under semi-convexity solutions. Indeed, in dimension five, for
\[
\kappa=\left(-\frac{11}{20},\frac{31}{80},\frac{31}{80},
\frac{31}{80},\frac{31}{80}\right),\qquad
z=\left(-\frac{49}{75},\frac{31}{75},\frac{31}{75},
\frac{31}{75},\frac{31}{75}\right),
\]
we have $H=1$, $\sigma_2(\kappa)=31/640>0$,
$\sum_i(H-\kappa_i)z_i=0$, and
\[
\frac{z_1^2+3\sum_{i\ne1}z_i^2-(\sum_i z_i)^2}
{\frac{H-\kappa_1}{H}(\sum_i z_i)^2}
=\frac{1072}{1125}<1.
\]
Thus even the coefficient $1$ fails for the unshifted form.
This motivates the shifted quantity
\[
b=\log(H+J_0).
\]
For bounded $H$, we can choose sufficiently  large $J_0$ to enlarge  the coefficient that must be supplied by the third-order quadratic form; for large $H$, the normalized bound $-K/H\to0$ recovers the convex case. Consequently, we obtain
\[
\Div_M(F\nabla b)
\geq cF(\nabla b,\nabla b)
+\Div_M\!\left(\frac{\nabla_M f}{H+J_0}\right)-CH.
\]
Here only $Df$ appears and the error $CH$ becomes bounded after the parallel change of variables.

The paper is organized as follows. Section~\ref{sec:spectral} establishes
the spectral estimates for semi-convex $\Gamma_2$-data.
Section~\ref{sec:algebra} proves the shifted third-order quadratic-form
inequality, and Section~\ref{sec:jacobi} derives the corresponding strong
and weak trace--Jacobi inequalities together with the local Jacobi energy
estimate. Section~\ref{sec:LL} constructs the parallel-hypersurface
transformation, proves the required graph representation and uniform
ellipticity and applies the De Giorgi estimate to reduce the pointwise
bound to a weighted integral. The final section proves the weighted
$L^1$ closure and completes the proof of Theorem~\ref{thm:main}.

\section{Semi-convex spectral structure}\label{sec:spectral}

\noindent

Throughout the paper, we denote
\[\sigma_{k;i_1\cdots i_l}(\lambda):=\left.\sigma_k(\lambda)
\right|_{\lambda_{i_1}=\cdots=\lambda_{i_l}=0}\]
for \(1\leq k\leq n\).

\begin{lemma}\label{lem1}
Let $n\geq 2$. Assume that $\lambda \in \Gamma_2$ and $\sigma_2(\lambda)=f$. Then we have 
$$\sigma_{1;i}(\lambda)\geq \frac{f}{\sigma_1(\lambda)},\ \ \ i\geq 1.$$
\end{lemma}
\begin{proof}
  It follows that 
\begin{equation*}
  \begin{split}
     \sigma_{1;i}(\lambda) \sigma_1(\lambda)-\sigma _2(\lambda)=& (\lambda_i+\sigma_{1;i}(\lambda))\sigma_{1;i}(\lambda)-(\lambda_i \sigma_{1;i}(\lambda)+\sigma_{2;i}(\lambda)) \\
      =& \sigma_{1;i}^2(\lambda)-\sigma_{2;i}(\lambda)\\
      =& \sigma_{1;i}^2(\lambda)-\frac{1}{2}\left(\sigma_{1;i}^2(\lambda)-\sum_{j\ne i}\lambda_j^2\right)\\
      =&\frac12\left(\sigma_{1;i}^2(\lambda)+\sum_{j\ne i}\lambda_j^2\right)\geq 0.
  \end{split}
\end{equation*}
\end{proof}

\begin{lemma}[Large-trace spectral structure]\label{lem:large-trace-spectrum}
Let $n\geq 2$ and $K\geq 0$. Suppose that $\lambda\in \Gamma_2$ satisfies $\lambda_1\ge\cdots\ge\lambda_n\geq -K$ and $\sigma_2(\lambda)=f$ with $f_0\leq f\leq f_1$ for two constants $f_1\geq f_0>0$.  Then there exists  $S_0 >0$ depending only on
$n$, $K$, $f_0$ and $f_1$ such that whenever $\sigma_1(\lambda)\geq S_0$, we have 
$$\frac{\sigma_1(\lambda)}{n}\leq  \lambda_1\leq \sigma_{1}(\lambda),$$
$$-K\leq\lambda_i  \leq n \frac{2(f_1+(n-1)^2K^2)}{S_0}+(n-2)K,\qquad i\geq2,$$
\begin{equation}\label{eq3}
  \frac{f_0}{2\sigma_1(\lambda)}\leq  \sigma_{1;1}(\lambda)\leq  \frac{2n(f_1+(n-1)^2K^2)}{\sigma_1(\lambda)},
\end{equation}
$$\frac{1}{2}\sigma_1(\lambda)\leq \sigma_{1;i}(\lambda)\leq 2\sigma_1(\lambda),\qquad i\geq2.$$
In fact, we can take 
\begin{equation*}
  \begin{split}
    S_0&=\max \Big \{2n(n-1)K, \sqrt{\frac{n-2}{n-1} } \frac{2n(f_1+(n-1)^2K^2)}{\sqrt{f_0}},  \\
      & (n-2)K+\sqrt{4nf_1+(4n(n-1)^2+(n-2)^2)K^2}\Big \}.
  \end{split}
\end{equation*}

\end{lemma}

\begin{proof}
Set 
\begin{equation}\label{eq:p-mu-definition}
 \mu:=(\lambda_2,\ldots,\lambda_n),
 \qquad
 \sigma_2(\mu)=\sigma_{2;1}(\lambda).
\end{equation}
Since every component of $\mu$ is at least $-K$
and their sum is $\sigma_{1;1}(\lambda)$, for every $i\geq2$,
\begin{equation}\label{eq:tail-pointwise}
 -K\leq\lambda_i\leq \sigma_{1;1}(\lambda)+(n-2)K.
\end{equation}

We need to provide  estimates for $\sigma_2(\mu)$.  Let $\sigma_{1;1}^+(\lambda)$ be the sum of the positive components of $\mu$ and  $\sigma_{1;1}^-(\lambda)$ be  the absolute value of the sum of its negative components, that is, 
$$\sigma_{1;1}^+(\lambda)=\sum_{\substack{\lambda_i>0\\ i\geq 2}}\lambda_i,\qquad \sigma_{1;1}^-(\lambda)=-\sum_{\substack{\lambda_i<0\\ i\geq 2}}\lambda_i.$$
Then it is clear that 
\[  0\leq \sigma_{1;1}^-(\lambda) \leq(n-1)K\]
and therefore, by $\sigma_{1;1}(\lambda)=\sigma_{1;1}^+(\lambda)-\sigma_{1;1}^-(\lambda)$, 
we have 
$$\sigma_{1;1}^+(\lambda)=\sigma_{1;1}(\lambda)+\sigma_{1;1}^-(\lambda) \leq \sigma_{1;1}(\lambda)+(n-1)K.$$
Products of two positive components and products of two negative components
are nonnegative, whereas the sum of all mixed products equals $-\sigma_{1;1}^+(\lambda)  \sigma_{1;1}^-(\lambda)$.
Consequently, we get 
\begin{equation}\label{eq:q-lower-bound}
 \sigma_2(\mu)\geq-\sigma_{1;1}^+(\lambda)  \sigma_{1;1}^-(\lambda)
 \geq-(n-1)K\sigma_{1;1}(\lambda)-(n-1)^2K^2.
\end{equation}
 On the other hand, the Cauchy--Schwarz inequality gives
\[  |\mu|^2\geq\frac{\sigma_{1;1}^2(\lambda)}{n-1},
\]
and 
\begin{equation}\label{eq:q-upper-bound}
 \sigma_2(\mu)=\frac12(\sigma_{1;1}^2(\lambda)-|\mu|^2)
 \leq\frac{n-2}{2(n-1)}\sigma_{1;1}^2(\lambda).
\end{equation}
Splitting $\sigma_2(\lambda)$ with respect to $\lambda_1$ yields the exact
identity
\begin{equation}\label{eq:sigma2-splitting}
 f=\lambda_1 \sigma_{1;1}(\lambda)+ \sigma_2(\mu).
\end{equation}
This with \eqref{eq:q-lower-bound} leads to 
\[  (\lambda_1-(n-1)K)\sigma_{1;1}(\lambda) \leq f_1+(n-1)^2K^2. \]
Since $\lambda_1\geq \sigma_1(\lambda)/n$, by taking $\sigma_1(\lambda)\geq 2n(n-1)K$, we have $\lambda_1\geq 2(n-1)K$ and therefore 
\begin{equation}\label{eq:p-upper-first}
 \sigma_{1;1}(\lambda) \leq \frac{2(f_1+(n-1)^2K^2)}{\lambda_1}.
\end{equation}
Furthermore, we choose $\sigma_1(\lambda)/n>0$ such that 
$$\frac{n-2}{2(n-1)} \frac{4(f_1+(n-1)^2K^2)^2}{\lambda_1^2}\leq \frac{n-2}{2(n-1)} \frac{4n^2(f_1+(n-1)^2K^2)^2}{\sigma_1^2(\lambda)} \leq \frac{f_0}{2}. $$
 From
\eqref{eq:q-upper-bound} and \eqref{eq:sigma2-splitting}, we obtain
\begin{equation}\label{eq:lambda1p-lower}
 \lambda_1 \sigma_{1;1}(\lambda)=f-\sigma_2(\mu) \geq f_0-\frac{n-2}{2(n-1)}\sigma_{1;1}^2(\lambda)\geq\frac{f_0}{2}.
\end{equation}
In all, we can choose 
\begin{equation}\label{eq2}
  S_0=\max\left\{2n(n-1)K, \sqrt{\frac{n-2}{n-1} } \frac{2n(f_1+(n-1)^2K^2)}{\sqrt{f_0}}\right\},
\end{equation}
such that when $\sigma_1(\lambda)\geq S_0>0$ we have 
$$ \frac{f_0}{2}\leq \lambda_1 \sigma_{1;1}(\lambda)\leq  2(f_1+(n-1)^2K^2).$$
In particular, for every $\lambda_i$ ($i\geq2$), 
\eqref{eq:tail-pointwise} gives a uniform bound
\begin{equation}\label{eq1}
  \begin{split}
    -K\leq\lambda_i & \leq  \sigma_{1;1}(\lambda)+(n-2)K\\
      &  \leq \frac{2(f_1+(n-1)^2K^2)}{\lambda_1}+(n-2)K\\
      &\leq n \frac{2(f_1+(n-1)^2K^2)}{\sigma_{1}(\lambda)}+(n-2)K\\
      & \leq n \frac{2(f_1+(n-1)^2K^2)}{S_0}+(n-2)K.
  \end{split}
\end{equation}
From $\sigma_{1;1}(\lambda)>0$ and $\lambda_1\geq \sigma_1(\lambda)/n$, it follows that 
$$\frac{\sigma_1(\lambda)}{n}\leq  \lambda_1\leq \lambda_1+\sigma_{1;1}(\lambda)=\sigma_{1}(\lambda).$$
For $i\geq 2$, by \eqref{eq1} and $\sigma_{1;i}(\lambda)=\sigma_1(\lambda)-\lambda_i$, we see that 
$$\sigma_1(\lambda)-2n \frac{f_1+(n-1)^2K^2}{S_0}-(n-2)K\leq \sigma_{1;i}(\lambda)\leq \sigma_1(\lambda)+K.$$
Since $\sigma_1(\lambda)>S$, we have 
$$\left(1-2n\frac{f_1+(n-1)^2K^2}{S_0^2}-\frac{(n-2)K}{S_0}\right)\sigma_1(\lambda)\leq \sigma_{1;i}(\lambda)\leq \left(1+\frac{K}{S_0}\right)\sigma_1(\lambda).$$
If we choose $S>0$ such that 
$$2n\frac{f_1+(n-1)^2K^2}{S_0^2}+\frac{(n-2)K}{S_0}\leq \frac{1}{2},$$ then we get 
$$\frac{1}{2}\sigma_1(\lambda)\leq \sigma_{1;i}(\lambda)\leq 2\sigma_1(\lambda).$$
In view of \eqref{eq2},  taking 
\begin{equation*}
  \begin{split}
    S_0&=\max \Big \{2n(n-1)K, \sqrt{\frac{n-2}{n-1} } \frac{2n(f_1+(n-1)^2K^2)}{\sqrt{f_0}},  \\
      & (n-2)K+\sqrt{4nf_1+(4n(n-1)^2+(n-2)^2)K^2}\Big \}.
  \end{split}
\end{equation*}
we complete the proof.
\end{proof}

\section{The shifted semiconvex quadratic form}\label{sec:algebra}
\noindent

This section contains the central algebraic estimate.  
For a fixed $k\geq 1$, we let
\begin{equation}\label{eq:Dk}
 D_k:=\operatorname{diag}(3,\ldots,3),
 \quad (D_k)_{kk}=1,
 \qquad e:=(1,\ldots,1)^T.
\end{equation}
Thus, for $z\in\R^n$, define $$z^TD_kz:=z_k^2+3\sum_{i\ne k}z_i^2.$$

\begin{lemma}\label{lem:Rk}
Let $\lambda\in \Gamma_2$ and $\sigma_2(\lambda)=f$. Assume that $$\ell_i:=\sigma_1(\lambda)-\lambda_i,\ \ \ i=1,\ldots,n.$$
Then we have 
\begin{equation}\label{eq:R-def}
 R_k:=\max_{\substack{z\ne0\\ \ell\cdot z=0}}
      \frac{(e\cdot z)^2}{z^TD_kz}=\frac13\left[ n+2- \frac{\bigl((n-1)\sigma_1(\lambda)+2\ell_k\bigr)^2} {(n-1)\sigma_1^2(\lambda)-2f+2\ell_k^2} \right].
\end{equation}
\end{lemma}

\begin{proof}
We equip $\mathbb{R}^n$ with the inner product $\langle v,w\rangle_{D_k}=v^TD_kw$. 
Since the quotient is homogeneous of degree zero in $t$, the original
maximization problem is equivalent to
\[R_k=\max\left\{(e\cdot z)^2:\ \ell\cdot z=0,\quad z^T D_k z=1\right\}.
\]
Since\[ \ip{D_k^{-1}\ell}{z}=\ell\cdot z,\]
the vector $D_k^{-1}\ell$ is a $D_k$-normal vector to 
\[\mathcal H:=\left\{z\in\R^n:\ell\cdot z=0\right\}.
\]
Let
\[q:=D_k^{-1}e.\]
Then
\[e\cdot z=\ip{q}{z}.\]
Therefore, the maximum is the squared $D_k$-norm of the orthogonal
projection of $q$ onto $\mathcal H$, i.e.,
\[\dnorm{q}^{2}-\frac{\ip{q}{D_k^{-1}\ell}^{2}}
     {\dnorm{D_k^{-1}\ell}^{2}}.\]
Using $q=D_k^{-1}e$, this becomes
\[e^T D_k^{-1}e
-\frac{\bigl(e^T D_k^{-1}\ell\bigr)^2}
     {\ell^T D_k^{-1}\ell}.\]

Moreover, since $D_k^{-1}$ has diagonal entry $1$ in the $k$th
position and $1/3$ in every other position, we have
\begin{align*}
 e^TD_k^{-1}e&=\frac{n+2}{3},\\
 e^TD_k^{-1}\ell
 &=\frac13\left(\sum_i\ell_i+2\ell_k\right)
   =\frac{(n-1)\sigma_1(\lambda)+2\ell_k}{3},\\
 \ell^TD_k^{-1}\ell
 &=\frac13\left(\sum_i\ell_i^2+2\ell_k^2\right)
   =\frac{(n-1)\sigma_1^2(\lambda)-2f+2\ell_k^2}{3},
\end{align*}
where $\sum_i\ell_i^2=(n-1)\sigma_1^2(\lambda)-2f$ was used in the last equalities.
Hence the maximum equals
\[\frac{1}{3}\left[n+2-\frac{\bigl((n-1)\sigma_1(\lambda)+2\ell_k\bigr)^2}
     {(n-1)\sigma_1^2(\lambda)-2f+2\ell_k^2}\right].\]
\end{proof}

\begin{lemma}[Uniform shifted gap]\label{lem:shifted-gap}
Let $n\geq 2$ and $K\geq 0$. Suppose that $\lambda\in \Gamma_2$ satisfies $\lambda_1\ge\cdots\ge\lambda_n\geq -K$ and $\sigma_2(\lambda)=f$ with $f_0\leq f\leq f_1$ for two constants $f_1\geq f_0>0$ and that $$\ell_i:=\sigma_1(\lambda)-\lambda_i,\ \ \ i=1,\ldots,n.$$
Then there exist  some constants $S_0>0$ and $J_0>0$ depending only on $n$, $K$, $f_0$ and $f_1$ such that
\begin{align}
 1-\left(1+(1+\theta)\frac{\ell_k}{\sigma_1(\lambda)}\right)R_k &\geq \frac{1-\theta}{6} &&\mbox{if}\ \ \sigma_1(\lambda )\geq S_0\ \ \mbox{and}\ \ k\geq2 ,\label{eq:gap-kge2}\\
 1-\left(1+(1+\theta)\frac{\ell_1}{\sigma_1(\lambda)}\right)R_1
 &\geq\left(\frac13-\theta\right)\frac{\ell_1}{\sigma_1(\lambda)} &&\mbox{if}\ \ \sigma_1(\lambda )\geq S_0,  \label{eq:gap-k1}\\
  1-\left(1+(1+\theta)\frac{\ell_k}{\sigma_1(\lambda)+J_0}\right)R_k& \geq \frac{\delta}{2} &&\mbox{if}\ \ \sigma_1(\lambda )\leq S_0 \ \ \mbox{and}\ \ k\geq1,
  \label{eq:gap-k3}
\end{align}
where $\theta\in (0,1/3)$ and $\delta=2f_0/S_0^2\in(0,1)$.
\end{lemma}

\begin{proof}
\textbf{Case 1. $\sigma_1(\lambda)\geq S_0>0$.}
By Lemma~\ref{lem:large-trace-spectrum}, if $k\geq2$, then
\[
 \ell_k=\sigma_1(\lambda)-\lambda_k,\qquad |\lambda_k|\leq C.
\]
Consequently, we have 
\[\begin{aligned}
L_k &:=(n-1)\sigma_1^2(\lambda)-2f+2\ell_k^2 \\
&=(n-1)\sigma_1^2(\lambda)-2f+2(\sigma_1(\lambda)-\lambda_k)^2\\
&=(n+1)\sigma_1^2(\lambda)   -4\sigma_1(\lambda) \lambda_k+2\lambda_k^2-2f.
\end{aligned}\]
and 
$$
\begin{aligned}
N_k&:=\bigl((n-1)\sigma_1(\lambda)+2\ell_k\bigr)^2\\
&=\bigl((n+1)\sigma_1(\lambda)-2\lambda_k\bigr)^2\\
&=(n+1)^2\sigma_1^2(\lambda)   -4(n+1)\sigma_1(\lambda) \lambda_k+4\lambda_k^2.
\end{aligned}$$
Therefore, \eqref{eq:R-def}  gives
\[  R_k=\frac{1}{3}\left(n+2-\frac{N_k}{L_k}\right)\]
and \[
 \left(1+(1+\theta)\frac{\ell_k}{\sigma_1(\lambda)}\right)R_k
 =\frac{1}{3}\left(1+(1+\theta)\frac{\sigma_1(\lambda)+\lambda_k}{\sigma_1(\lambda)}\right)\left(n+2-\frac{N_k}{L_k}\right).
\]
Letting $\sigma_1(\lambda) \rightarrow \infty$, we have 
$$1-\left(1+(1+\theta)\frac{\ell_k}{\sigma_1(\lambda)}\right)R_k \rightarrow \frac{1-\theta}{3}.$$
This implies that there exists some constant $S_0>0$ depending only on $n$, $K$, $f_0$ and $f_1$ such that when $\sigma_1(\lambda)>S_0$, we have 
$$1-\left(1+(1+\theta)\frac{\ell_k}{\sigma_1(\lambda)}\right)R_k \geq \frac{1-\theta}{6}.$$
Direct computation gives 
\begin{equation}\label{eq:one-minus-R1}
 1-R_1=
 \frac{4(n-1)\sigma_1(\lambda) \ell_1+2(n-1)f-(2n-6)\ell_1^2}
 {3\bigl((n-1)\sigma_1^2(\lambda)-2f+2\ell_1^2\bigr)}.
\end{equation}
Choosing sufficiently large $S_0>0$, by $\sigma_{1;1}(\lambda)\leq  \frac{2n(f_1+(n-1)^2K^2)}{\sigma_1(\lambda)}$ in Lemma~\ref{lem:large-trace-spectrum}, we know that 
$$ \sigma_{1;1}^2(\lambda)\leq   f_0$$
and $$(n-3)\sigma_{1;1}^2(\lambda)\le(n-1)f_0.$$
Thus 
$$4(n-1)\sigma_1(\lambda) \ell_1+2(n-1)f-(2n-6)\ell_1^2\geq 4(n-1)\sigma_1(\lambda) \ell_1$$
and 
$$(n-1)\sigma_1^2(\lambda)-2f+2\ell_1^2\leq (n-1)\sigma_1^2(\lambda),$$
which immediately leads to 
$$1-R_1\geq \frac{4}{3}\frac{\ell_1}{\sigma_1(\lambda)}>0.$$
In particular, $R_1\leq1$ and therefore
\begin{align*}
 1-\left(1+(1+\theta)\frac{\ell_1}{\sigma_1(\lambda)}\right)R_1
 &=(1-R_1)-(1+\theta)\frac{\ell_1}{\sigma_1(\lambda)}R_1\\
 &\geq\left(\frac13-\theta\right)\frac{\ell_1}{\sigma_1(\lambda)}.
\end{align*}
\textbf{Case 2. $0<\sigma_1(\lambda)\leq S_0$.}
Since $\ell\cdot z=0$  and $\ell=\sigma_1(\lambda)e-\lambda$, we have 
\[\sigma_1(\lambda)(e\cdot z)=\lambda\cdot z.
\]
Hence
\begin{equation}\label{eq:compact-trace-gap}
 (e\cdot z)^2  \leq\frac{|\lambda|^2}{\sigma_1^2(\lambda)}|z|^2
 =\left(1-\frac{2f}{\sigma_1^2(\lambda)}\right)|z|^2
 \leq(1-\delta)|z|^2
\end{equation}
with $\delta=2f_0/S_0^2\in(0,1)$.  Since
$z^TD_kz\geq|z|^2$, $R_k\leq1-\delta$.  
 Choosing $J_0>0$ such that 
\begin{equation}\label{eq:J0-explicit}
 J_0\geq\frac{4(1+\theta)S_0}{\delta},
\end{equation}
by $0<\ell_k=\sigma_1(\lambda)-\lambda_k<2\sigma_1(\lambda)\leq2S_0$, we obtain
\[  (1+\theta)\frac{\ell_k}{\sigma_1(\lambda)+J_0}R_k
 \leq\frac{2(1+\theta)S_0}{J_0}\leq\frac\delta2,
\]
and therefore
\[
 1-\left(1+(1+\theta)\frac{\ell_k}{\sigma_1(\lambda)+J_0}\right)R_k=(1-R_k)-(1+\theta)\frac{\ell_k}{\sigma_1(\lambda)+J_0}R_k
 \geq\frac{\delta}{2}.
\]
\end{proof}

\begin{lemma}[Shifted trace quadratic form]\label{lem:quadratic}
Let $n\geq 2$ and $K\geq 0$. Suppose that $\lambda\in \Gamma_2$ satisfies $\lambda_1\ge\cdots\ge\lambda_n\geq -K$ and $\sigma_2(\lambda)=f$ with $f_0\leq f\leq f_1$ for two constants $f_1\geq f_0>0$ and that $$\ell_i:=\sigma_1(\lambda)-\lambda_i,\ \ \ i=1,\ldots,n.$$
Then there exist  some constants $C>0$ and $J_0>0$ depending only on $n$, $K$, $f_0$ and $f_1$ such that  
\begin{equation}\label{eq:quadratic}
 z_k^2+3\sum_{i\ne k}z_i^2
 -\left[1+\frac{7}{6}\frac{\ell_k}{\sigma_1(\lambda)+J_0}\right]
       \left(\sum_iz_i\right)^2
 \geq-C\left(\sum_i\ell_iz_i\right)^2
\end{equation}
 for  any $1\leq k\leq n$ and $z=(z_1,\ldots,z_n)\in\R^n$.
\end{lemma}

\begin{proof}
Write $D=D_k$, $e=(1,\ldots,1)^T$, $a=1+\frac{7}{6}\frac{\ell_k}{\sigma_1(\lambda)+J_0}$ and
\[\mathbf M:=D-aee^T.\]
Then the left-hand side of \eqref{eq:quadratic} is $z^T\mathbf Mz$.

First we impose the homogeneous constraint $\ell\cdot z=0$.  By the definition
of $R_k$, we have 
\[
 (e\cdot z)^2\leq R_k\,z^TDz.
\]
Hence
\begin{equation}\label{eq:homogeneous-positive}
 z^T\mathbf Mz=z^TDz-a(e\cdot z)^2
 \geq(1-aR_k)z^TDz\geq0
\end{equation}
following from Lemma~\ref{lem:shifted-gap}.  Thus $\mathbf M$ is positive definite when
restricted to the hyperplane $\ell^\perp$.

We next remove the homogeneous constraint.  Let
\[\tau:=\ell\cdot z,
 \qquad A_0:=e^TD^{-1}e,
 \qquad B:=e^TD^{-1}\ell,
 \qquad C_\ell:=\ell^TD^{-1}\ell.
\]
Since $n\geq2$, we have 
\[
 A_0=\frac{n+2}{3}>1,
 \qquad a\geq1,
 \qquad aA_0-1\geq A_0-1=\frac{n-1}{3}>0.
\]
The Sherman--Morrison formula gives
\begin{equation}\label{eq:SM}
 \mathbf M^{-1}=D^{-1}+\frac{a}{1-aA_0}
                D^{-1}ee^TD^{-1}.
\end{equation}
Therefore
\begin{align}
 d_k:=\ell^T\mathbf M^{-1}\ell
 &=C_\ell+\frac{aB^2}{1-aA_0}\notag\\
 &=\frac{C_\ell(1-aA_0)+aB^2}{1-aA_0}
 =\frac{C_\ell(1-aR_k)}{1-aA_0}
 =-\frac{C_\ell(1-aR_k)}{aA_0-1}<0,
 \label{eq:dk}
\end{align}
where $R_k=A_0-B^2/C_\ell$ was used.

For completeness, minimizing $z^T\mathbf Mz$ subject to $\ell\cdot z=\tau$ gives the
Lagrange multiplier equations
\[ 2\mathbf Mz=\omega\ell,
 \qquad \ell\cdot z=\tau\qquad \mbox{for}\ \ \ \omega \in\mathbb{R}.
\]
Solving them with \eqref{eq:SM} yields
\begin{equation}\label{eq:constrained-minimizer}
 z_{0}=\frac{\tau}{d_k}\mathbf M^{-1}\ell,
 \qquad
 \min_{\ell\cdot z=\tau}z^T\mathbf Mz=\frac{\tau^2}{d_k}.
\end{equation}
This stationary point is the global constrained minimum because the
quadratic form is positive definite on the tangent hyperplane
$\ell^\perp$, as shown in \eqref{eq:homogeneous-positive}.  More precisely,
if $z_0=(\tau/d_k)\mathbf M^{-1}\ell$ and $v\in\ell^\perp$, then $\mathbf Mz_0$ is a
multiple of $\ell$ and hence
\[
 (z_0+v)^T\mathbf M(z_0+v)=z_0^T\mathbf Mz_0+v^T\mathbf Mv\geq z_0^T\mathbf Mz_0.
\]

It remains to prove a uniform lower bound for $-d_k$.  The denominator
$aA_0-1$ is bounded above. Indeed, from $0<\ell_k=\sigma_1(\lambda)-\lambda_k<2\sigma_1(\lambda)$, it follows that 
\[ 1\leq a\leq1+2(1+\theta),
\]
which implies that 
\begin{equation}\label{eq4}
  aA_0-1 \le
\left[1+2(1+\theta)\right]\frac{n+2}{3}-1.
\end{equation}
 On the compact region $0<\sigma_1(\lambda) \leq S_0$ from Lemma~\ref{lem:shifted-gap}, we know that 
 $$1-aR_k\geq \frac{f_0}{S_0^2}>0.$$
Using the Newton-Maclaurin inequality  and the Cauchy–Schwarz inequality, we have 
\begin{equation*}
  \begin{split}
     C_\ell& \geq\frac13|\ell|^2 \\
      & \geq \frac{1}{3n}\left(\sum_i\ell_i\right)^2\\
      &\geq \frac{(n-1)^2}{3n} \sigma_1^2(\lambda)\\
      & \geq \frac{(n-1)^2}{3n}\frac{2n}{n-1}f_0=\frac{2(n-1)}3f_0.
  \end{split}
\end{equation*} 
Hence 
$$-d_k\geq \frac{2(n-1)}{3S_0^2} \frac{f_0^2 }{\left[1+2(1+\theta)\right]\frac{n+2}{3}-1}.$$
On the large-trace region $\sigma_1(\lambda) \geq S_0$, it is clear that
\[ \frac{(n-1)^2}{3n}\sigma_1^2(\lambda) \leq C_\ell=\frac13\left((n-1)\sigma_1^2(\lambda)-2f+2\ell_k^2\right)\leq 
\frac{n+7}{3}\sigma_1^2(\lambda).
\]
If $k\geq2$, Lemma~\ref{lem:shifted-gap} gives
$1-aR_k\geq \frac{1-\theta}{6}$ and therefore 
$$-d_k\geq S_0^2 \frac{(n-1)^2}{3n}\frac{1-\theta}{6}\frac{1}{\left[1+2(1+\theta)\right]\frac{n+2}{3}-1}.$$
  If $k=1$, then
\eqref{eq:gap-k1} and \eqref{eq3} give
\[ C_\ell(1-aR_1)\geq \frac{(n-1)^2}{3n}\sigma_1^2(\lambda)\left(\frac13-\theta\right)\frac{\ell_1}{\sigma_1(\lambda)}
\geq \frac{(n-1)^2}{3n}\left(\frac13-\theta\right)\frac{f_0}{2}.\]
This together with \eqref{eq4} yields 
$$-d_k\geq \frac{(n-1)^2}{3n}\left(\frac13-\theta\right)\frac{f_0}{2}\frac{1}{\left[1+2(1+\theta)\right]\frac{n+2}{3}-1}.$$
Thus  \eqref{eq:constrained-minimizer}  implies that there exists some constant $C>0$ depending only on $n$, $K$, $f_0$ and $f_1$ such that 
\[ z^T\mathbf Mz\geq-C\tau^2,
\]
which yields  \eqref{eq:quadratic}.
\end{proof}

\begin{proposition}[Tensor form]\label{prop:tensor}
Let $n\geq 2$ and $K\geq 0$. Suppose that $\lambda\in \Gamma_2$ satisfies $\lambda_1\ge\cdots\ge\lambda_n\geq -K$ and $\sigma_2(\lambda)=f$ with $f_0\leq f\leq f_1$ for two constants $f_1\geq f_0>0$ and that $$\ell_i:=\sigma_1(\lambda)-\lambda_i,\ \ \ i=1,\ldots,n.$$
Let $\mathcal T=(\mathcal T_{ijk})$ be a completely symmetric three-tensor and define
\[ s_k:=\sum_i\mathcal T_{iik},
 \qquad q_k:=\sum_i\ell_i\mathcal T_{iik},
\]
then there exist  some constants $C>0$ and $J_0>0$ depending only on $n$, $K$, $f_0$ and $f_1$ such that  
\begin{equation}\label{eq:tensor}
 \sum_{i,j,k}\mathcal T_{ijk}^2-\sum_k s_k^2
 \geq\frac{7}{6}\frac{1}{\sigma_1(\lambda)+J_0}\sum_k \ell_k s_k^2
      -C\sum_kq_k^2.
\end{equation}
\end{proposition}

\begin{proof}
Complete symmetry gives
\begin{equation}\label{eq:tensor-reduction}
 \sum_{i,j,k}\mathcal T_{ijk}^2 \geq
 \sum_k\left(\mathcal T_{kkk}^2+3\sum_{i\ne k}\mathcal T_{iik}^2\right).
\end{equation}
For each fixed $k\geq 1$ and $\theta=1/6$,  applying Lemma~\ref{lem:quadratic} to $z_i=\mathcal T_{iik}$ yields 
$$z_{k}^2+3\sum_{i\ne k}z_{i}^2-s_k^2\geq \frac{7}{6}\frac{1}{\sigma_1(\lambda)+J_0} \ell_k s_k^2
      -Cq_k^2.$$
Summing over $k$ and using \eqref{eq:tensor-reduction} proves 
\eqref{eq:tensor}.
\end{proof}

\section{Strong and weak trace--Jacobi inequalities}\label{sec:jacobi}
\noindent

Let \(X(x)=(x,u(x))\) be the position vector on \(M=\{(x,u(x)):x\in B_{10}\}\) and denote the outer (downward) unit normal of \(M\) by
\[\nu=\frac{1}{W}(Du,-1),\qquad
W=\sqrt{1+|Du|^2}.\]
Let \(\{E_1,\ldots,E_n,E_{n+1}\}\) be the standard orthonormal
coordinate frame of \(\mathbb{R}^{n+1}\). Sometimes we may choose an orthonormal principal frame $\{e_1,e_2,\ldots,e_n,\nu\}$
in \(\mathbb{R}^{n+1}\) such that \(\{e_1,e_2,\ldots,e_n\}\) are
tangent to \(M\) and \(\nu\) is the unit normal on \(M\). We continue to regard the prescribed function $f(x)$ as the pullback
$f\circ\pi$ to the graph under the horizontal projection
$\pi(X(x))=x$ and suppress the pullback symbol.
We use \(\nabla\) to denote the Levi--Civita connection on \(M\).
For any smooth function \(v\) on \(M\), we denote
\[v_i=\nabla_{e_i}v,
\qquad
v_{ij}=\nabla^2 v(e_i,e_j).\]
The following fundamental equations for hypersurfaces in
\(\mathbb{R}^{n+1}\) are well known:\[
\begin{aligned}
X_i &= e_i,\ \ 
X_{ij} = -h_{ij}\nu &&\text{(Gauss formula)},\\
\nu_i &= h_{ij}e_j
&&\text{(Weingarten equation)},\\
h_{ijk} &= h_{ikj}
&&\text{(Codazzi equation)},\\
R_{ijkl} &= h_{ik}h_{jl}-h_{il}h_{jk}
&&\text{(Gauss equation)}.
\end{aligned}
\]
Here \(h_{ij}\) is the second fundamental form of \(M\),
\[h_{ijk}=\nabla_{e_k}h_{ij},
\]
and \(R_{ijkl}\) is the curvature tensor of \(M\). We also have
the following commutator formula:
\[h_{ijkl}-h_{ijlk}
=\sum_m h_{im}R_{mjkl}+\sum_m h_{mj}R_{mikl}.
\]
Combining this with the Gauss and Codazzi equations, we have
\begin{equation}\label{eq5}
h_{iikk}=h_{kkii}
+\sum_m\left(h_{mi}^{\,2}h_{kk}-h_{mk}^{\,2}h_{ii}\right).
\end{equation}
In this orthonormal principal frame, the Weingarten map (shape operator) $A$ satisfies $$Ae_i=\kappa_i e_i$$ and then 
\[
h_{ij}=\langle Ae_i,e_j\rangle=\kappa_i\delta_{ij}
\]
for $\kappa=(\kappa_1,\ldots,\kappa_n)$. Thus, the scalar curvature equation becomes
\[\sigma_2(\kappa)=\sigma_2\bigl(\lambda(h_{ij})\bigr)=
\frac{1}{2}\left[\left(\sum_i h_{ii}\right)^2-\sum_{i,j}h_{ij}^2\right]
=\frac{1}{2}\left(H^2-|A|^2\right)=f,\]
where \(H\) and \(|A|\) denote the mean curvature and the norm of the
second fundamental form of \(M\), respectively. 
Taking the covariant derivative of the curvature equation with respect to $e_k$, we have 
$$F_{ij}h_{ijk}=f_k,$$
where the linearized operator 
$$F_{ij}:=\frac{\partial \sigma_2}{\partial h_{ij}}=H\delta_{ij}-h_{ij}.$$
The  Codazzi equations imply
\begin{equation}\label{eq:divF}
\nabla_i F_{ij}=0.
\end{equation}
Thus 
$$F_{ij}\nabla_i \nabla_j v=\Div_M(F\nabla v).$$
Differentiating the curvature equation twice gives 
$$F_{ij}h_{ijkk}-\sum_{i,j}h_{ijk}^2+\left(\sum_{i}h_{iik}\right)^2=f_{kk}.$$
We apply \eqref{eq5} to obtain
\begin{equation*}
  \begin{split}
    \sum_iF_{ii}H_{ii} & =\sum_{k,i}F_{ii}h_{kkii}=\sum_{k,i} F_{ii}(h_{iikk}+h_{kk}^2h_{ii}-h_{ii}^2h_{kk}) \\
      & =\Delta_Mf+\sum_{i,j,k}h_{ijk}^2-\sum_{k}\left(\sum_{i}h_{iik}\right)^2+\sum_{k,i} F_{ii}(h_{kk}^2h_{ii}-h_{ii}^2h_{kk}).
  \end{split}
\end{equation*}
By the aid of 
\begin{equation}\label{eq:newton-contractions}
 \sum_i F_{ii} h_{ii}=2\sigma_2=H^2-|A|^2,
 \qquad
 \sum_i F_{ii}h_{ii}^2
      =H|A|^2-\sum_ih_{ii}^3,
\end{equation}
we have 
\begin{align*}
 \sum_{i,k}F_{ii}(h_{kk}^2h_{ii}-h_{ii}^2h_{kk})
 &=|A|^2\sum_iF_{ii} h_{ii}
       -H\sum_iF_{ii}h_{ii}^2\\
 &=|A|^2(H^2-|A|^2)
       -H\left(H|A|^2-\sum_ih_{ii}^3\right)\\
 &=H\sum_i\kappa_i^3-\left(\sum_i\kappa_i^2\right)^2
       =\mathcal C(\kappa).
\end{align*}
It follows that 
\begin{equation}\label{eq:LFH-exact}
F_{ij}H_{ij}
 =\Delta_Mf+|\nabla A|^2-|\nabla H|^2+\mathcal C(\kappa).
\end{equation}
In fact, we write 
\begin{equation}\label{eq:C-detailed}
 \mathcal C(\kappa)
 : =\sum_{i<j}\kappa_i\kappa_j(\kappa_i-\kappa_j)^2.
\end{equation}

The commutator  is nonnegative for $\kappa\geq 0$, but not for  $\kappa \in \Gamma_2$.  The next lemma gives a lower bound for  $\mathcal C(\kappa)$.

\begin{lemma}\label{lem:commutator}
Let $n\geq 2$ and $K\geq 0$. Suppose that  $\kappa\in \Gamma_2$ satisfies $\kappa_1\geq\ldots \geq \kappa_n\geq -K$ and $\sigma_2(\kappa)=f$ with $f_0\leq f \leq f_1$ for two constants $f_1\geq f_0>0$.  Then there exists some constant $C>0$ depending only on $n$, $K$, $f_0$ and $f_1$ such that
\begin{equation}\label{eq:C-lower-detailed}
\mathcal C(\kappa)\geq-C H^2.
\end{equation}
\end{lemma}

\begin{proof}
Let $S$ be a large-trace threshold
for which Lemma~\ref{lem:large-trace-spectrum} holds.

If $H\leq S$, then $-K\leq\kappa_i<H\leq S$  following from $H-\kappa_i>0$ in $\Gamma_2$.  Hence  we have $|\mathcal C(\kappa)|\leq C$ for some constant $C>0$ depending only on $n$, $K$, $f_0$ and $f_1$.  On the
other hand, the Newton--Maclaurin inequality  gives
\begin{equation}\label{eq6}
  H^2\geq\frac{2n}{n-1}\sigma_2(\kappa)
       \geq\frac{2n}{n-1}f_0.
\end{equation}
Thus we have the lower bound 
$$\mathcal C(\kappa)\geq -C\geq -C\frac{n-1}{2nf_0}H^2.$$
Next, we consider $H>S$.  Set
\[  A_m:=\sum_{i=2}^n\kappa_i^m,\ \ \ m\geq 1.\]
From the bound for $\kappa$ in Lemma~\ref{lem:large-trace-spectrum}, we have  $|A_1|+|A_2|+|A_3|\leq C$ for some constant $C>0$ depending only on $n$, $K$, $f_0$ and $f_1$.  Since
$H=L+A_1$, direct expansion gives
\begin{align}
 \mathcal C(\kappa)
 &=(\kappa_1+A_1 )(\kappa_1^3+A_3)-(\kappa_1^2+A_2)^2\notag\\
 &=A_1 \kappa_1^3-2A_2\kappa_1^2+\kappa_1A_3+A_1 A_3-A_2^2.             \label{eq:C-expanded-detailed}
\end{align}
The leading term $A_1 L^3$ is nonnegative.  The remaining terms satisfy
\[ -2A_2\kappa_1^2+\kappa_1A_3+A_1 A_3-A_2^2
 \geq-C\kappa_1^2-C\kappa_1-C\geq-CH^2,\]
where   Lemma~\ref{lem:large-trace-spectrum} and \eqref{eq6} were used.  This
proves \eqref{eq:C-lower-detailed} in the large-trace region.
\end{proof}

\begin{proposition}[Strong shifted trace--Jacobi inequality]
\label{prop:strong-jacobi}
Let $n\geq 2$, $K\geq 0$, $u\in C^4(B_{10})$ and $f\in C^2(B_{10})$ with $D^2u\geq -KI$ and $\inf_{B_{10}}f >0$. If $M=\{(x,u(x)):x\in B_{10}\}\subset\R^{n+1}$ is a smooth 2-convex graph satisfying the scalar curvature equation \eqref{eq:intro-main}, then there exist some constants $J_0>1$ and $C>0$ depending only on $n$, $K$, $\inf_{B_{10}}f$ and $\sup_{B_{10}}f$ such that
\begin{equation}\label{eq:strong-b-detailed}
F_{ij}b_{ij}\geq \frac{1}{6} F_{ij}b_ib_j
       +\frac{\Delta_Mf}{H_0}
       -C\frac{|\nabla_Mf|^2}{h}-CH,
\end{equation}
where
\begin{equation}\label{eq:h-b-definitions}
  H_0:=H+J_0,
 \qquad   b:=\log H_0.
\end{equation}
\end{proposition}

\begin{proof}
Choose an orthogonal principal frame such that the second fundamental form is diagonal, i.e., $h_{ij}=\kappa_i\delta_{ij}$. Thus, the quantities 
\[|\nabla A|^2,\qquad
|\nabla H|^2,\qquad
F_{ij}H_iH_j,\qquad
|\nabla_M f|^2
\]
are invariant.
Set
\[ T_{ijk}:=h_{ijk}.
\]
In view of Codazzi equation, this tensor is completely
symmetric.  The quantities in Proposition~\ref{prop:tensor} become
\begin{align}
 s_k&=\sum_iT_{iik}=\sum_i h_{iik}=H_k,                  \label{eq:sk-Hk}\\
 q_k&=\sum_iF_{ii}T_{iik}
     =\sum_iF_{ii} h_{iik}=f_k.                         \label{eq:qk-fk}
\end{align}
Therefore, Proposition~\ref{prop:tensor}  yields
\begin{equation}\label{eq:third-order-lower-detailed}
 |\nabla A|^2-|\nabla H|^2
 \geq\frac{1}{H_0}\frac{7}{6} F_{ij}H_iH_j
      -C|\nabla_Mf|^2.
\end{equation}
Then, by \eqref{eq:LFH-exact} and Lemma \ref{lem:commutator}, we have 
\begin{equation}\label{eq:strong-H-detailed}
F_{ij}H_{ij}
 \geq \frac{1}{H_0}\frac{7}{6} F_{ij}H_iH_j +\Delta_Mf-C|\nabla_Mf|^2-CH^2.
\end{equation}
The chain rule gives
\begin{equation}\label{eq:chain-rule-b}
 F_{ij}b_{ij}=\frac{F_{ij}H_{ij}}{H_0}-\frac{F_{ij}H_iH_j}{H_0^2},
 \qquad
 F_{ij}b_ib_j=\frac{F_{ij}H_iH_j}{H_0^2}.
\end{equation}
Dividing \eqref{eq:strong-H-detailed} by $H_0$ and using 
\eqref{eq:chain-rule-b} and  $H^2/H_0\leq H$, we obtain the Jacobi inequality
\eqref{eq:strong-b-detailed}.
\end{proof}

\begin{proposition}[Weak shifted Jacobi inequality]
\label{prop:weak-jacobi}
Let $n\geq 2$, $K\geq 0$, $u\in C^4(B_{10})$ and $f\in C^2(B_{10})$ with $D^2u\geq -KI$ and $\inf_{B_{10}}f >0$. If $M=\{(x,u(x)):x\in B_{10}\}\subset\R^{n+1}$ is a smooth 2-convex graph satisfying the scalar curvature equation \eqref{eq:intro-main}, then there exists some constant  $C>0$ depending only on $n$, $K$, $\inf_{B_{10}}f$, $\sup_{B_{10}}f$ and $\|Df\|_{L^{\infty}(B_{10})}$ such that
\begin{equation}\label{eq:weak-jacobi-detailed}
 \Div_M(F\nabla b)
 \geq \frac{1}{12} F_{ij}b_ib_j+\Div_M\!\left(\frac{\nabla_Mf}{H_0}\right)-CH
\end{equation}
in the sense of distributions, where $b$ and $H_0$ come from Proposition \ref{prop:strong-jacobi}. 
\end{proposition}

\begin{proof}
 Applying the product rule yields
\begin{align}
 \Div_M\!\left(\frac{\nabla_Mf}{H_0}\right)
 &=\frac{\Delta_Mf}{H_0}
   +\left\langle\nabla_Mf,\nabla_M\frac1H_0\right\rangle \notag\\
 &=\frac{\Delta_Mf}{H_0}
   -\frac{\langle\nabla_Mf,\nabla_Mb\rangle}{H_0}.
                                                               \label{eq:data-product-rule}
\end{align}
Hence
\begin{equation}\label{eq:data-divergence-detailed}
 \frac{\Delta_Mf}{H_0}
 =\Div_M\!\left(\frac{\nabla_Mf}{H_0}\right)
  +\frac{\langle\nabla_Mf,\nabla_Mb\rangle}{H_0}.
\end{equation}
Using weighted Young's inequality with $\varepsilon=1/12$ and $\tr(F^{-1})\leq \frac{nH}{\inf_{B_{10}}f}$ from Lemma \ref{lem1}, we obtain
\begin{equation}\label{eq:weighted-young-data}
  \begin{split}
\left|\frac{\langle\nabla_Mf,\nabla_Mb\rangle}{H_0}\right|&=\left|\frac{\langle F^{-1/2} \nabla_Mf,F^{1/2}\nabla_Mb\rangle}{H_0}\right|\\
 &\leq \frac{1}{12} F_{ij}b_ib_j
       +\frac{3}{ h^2}   (F^{-1})_{ij}f_if_j    \\
 &\leq    \frac{1}{12} F_{ij}b_ib_j  +\frac{3n}{ \inf_{B_{10}}f \ H_0}
   |\nabla_Mf|^2.   
  \end{split}
\end{equation}
 Since the Newton--Maclaurin inequality \eqref{eq6}, we can estimate 
 $$ \frac1H_0\leq\frac{n-1}{2n\inf_{B_{10}}f}H.$$
 Combining this with \eqref{eq:strong-b-detailed}, \eqref{eq:data-divergence-detailed} and \eqref{eq:weighted-young-data}, we obtain
 \begin{equation*}
 \begin{split}
    \Div_M(F\nabla b)
 &\geq \frac{1}{12} F_{ij}b_ib_j+\Div_M\!\left(\frac{\nabla_Mf}{H_0}\right)-\left(C+\frac{3n}{ \inf_{B_{10}}f}\right)\frac{ |\nabla_Mf|^2}{H_0}-CH \\
     & \geq \frac{1}{12} F_{ij}b_ib_j+ \Div_M\!\left(\frac{\nabla_Mf}{H_0}\right)-\left(C+\frac{3n}{\inf_{B_{10}}f}\right) \frac{n-1}{2n\inf_{B_{10}}f}|\nabla_Mf|^2H-CH
 \end{split}
\end{equation*}
for some  constant  $C>0$ depending only on $n$, $K$, $\inf_{B_{10}}f$, $\sup_{B_{10}}f$ and $\|Df\|_{L^{\infty}(B_{10})}$

Equivalently, every nonnegative $\phi\in C_c^1(M)$ satisfies
\begin{align} \label{eq:weak-test}
 -\int_M F_{ij} b_i\phi_j\,d\mu
\geq \frac{1}{12} \int_M F_{ij}b_ib_j\phi\,d\mu-\int_M \frac{1}{H_0}f_i\phi_i\,d\mu
 -C\int_M H\phi\,d\mu.            
\end{align}
\end{proof}

\begin{proposition}[Local Jacobi energy]\label{prop:energy}
Let $n\geq 2$, $K\geq 0$, $u\in C^4(B_{10})$ and $f\in C^2(B_{10})$ with $D^2u\geq -KI$ and $\inf_{B_{10}}f >0$. Assume that  $M=\{(x,u(x)):x\in B_{10}\}\subset\R^{n+1}$ is a smooth 2-convex graph satisfying the scalar curvature equation \eqref{eq:intro-main}. If $\eta\in C_c^1(B_2)$ and $0\leq\eta\leq1$, then there exists some constant  $C>0$ depending only on $n$, $K$, $\inf_{B_{10}}f$, $\sup_{B_{10}}f$, $\|Df\|_{L^{\infty}(B_{10})}$, $\|Du\|_{L^{\infty}(B_{2})}$ and $\|D\eta\|_{L^{\infty}(B_{2})}$ such that
\begin{equation}\label{eq:energy-fixed}
 \int_{M}\eta^2F_{ij}b_ib_j\,d\mu
 \leq C,
\end{equation}
where $b$ comes from Proposition \ref{prop:strong-jacobi}. 
\end{proposition}

\begin{proof}
Take $\phi=\eta^2$ in \eqref{eq:weak-test}.  Since
$\nabla(\eta^2)=2\eta\nabla\eta$, we obtain
\begin{align}
 \frac{1}{12}\int_{M}\eta^2F_{ij}b_ib_j\,d\mu
 &\leq-2\int_{M}\eta F_{ij}b_i\eta_j\,d\mu\notag\\
 &\quad+2\int_{M}\frac{1}{H_0} \eta f_i\eta_i\,d\mu
      +C\int_{M}H\eta^2\,d\mu.             \label{eq:energy-expanded}
\end{align}
Since $F_{ij}$ is positive definite,  Young's inequality with $\epsilon=\frac{1}{24}$ yields
\[ 2
 |\eta F_{ij}b_i\eta_j |
 \leq \frac{1}{24}\eta^2 F_{ij}b_ib_j+24F_{ij}\eta_i\eta_j,
\]
which implies that 
\begin{align*}
 2\left|\int_{M}\eta F_{ij}b_i\eta_j\,d\mu\right|
 \leq\frac{1}{24}\int_{M}\eta^2 F_{ij}b_ib_j\,d\mu
+24\int_{M}F_{ij}\eta_i\eta_j\,d\mu.
\end{align*}
Furthermore, by $1/H_0\leq 1$ and $|\nabla_Mf|\leq \|Df\|_{L^{\infty}(B_{10})}$, we have 
\[ 2\left|\int_{M}\frac{1}{H_0} \eta f_i\eta_i\,d\mu\right|
 \leq 2\|Df\|_{L^{\infty}(B_{10})}\int_{M}\eta|\nabla_M\eta|\,d\mu.
\]
Absorbing the first term into the left-hand side of
\eqref{eq:energy-expanded} proves 
\begin{equation}\label{eq:energy-general}
\begin{split}
  \frac{1}{24}\int_{M}\eta^2F_{ij}b_ib_j\,d\mu & \leq 24\int_{M}F_{ij}\eta_i\eta_j\,d\mu \\
    &  +C\int_{M}H\eta^2\,d\mu
      +2\|Df\|_{L^{\infty}(B_{10})}\int_{M}\eta|\nabla_M\eta|\,d\mu.
\end{split}
\end{equation}
Observing  that $\tr F=(n-1)H$, we have 
$$F_{ij}\eta_i\eta_j\leq (n-1)H |\nabla_M\eta|^2$$
 Since $H>0$, $d\mu=\sqrt{1+|Du|^2}dx$, and
$H=\Div(Du/\sqrt{1+|Du|^2})$, we have 
\begin{equation}\label{eq:L1-s}
\begin{aligned}
 \int_{M\cap\{x\in B_{2}\}}H\,d\mu
 &=\int_{B_2} H  \sqrt{1+|Du|^2}\,dx  \\
 &\leq\|\sqrt{1+|Du|^2}\|_{L^\infty(B_2)}\int_{B_2} \Div \left (\frac{Du}{\sqrt{1+|Du|^2}}\right)\,dx\\
 &=\|\sqrt{1+|Du|^2}\|_{L^\infty(B_2)}\int_{ \partial B_2}\frac{Du}{\sqrt{1+|Du|^2}}\cdot\nu\,dS_x\\
 & \leq |\partial B_2|\|\sqrt{1+|Du|^2}\|_{L^\infty(B_2)}.
\end{aligned}
\end{equation}
 Hence 
 $$\int_{M}H\eta^2\,d\mu\leq |\partial B_2|\|\sqrt{1+|Du|^2}\|_{L^\infty(B_2)}$$
 and
  \begin{equation*}
   \begin{split}
     \int_{M}F_{ij}\eta_i\eta_j\,d\mu & \leq (n-1)\int_{M\cap\{x\in B_{2}\}}H |\nabla_M\eta|^2\,d\mu \\
       & \leq (n-1)|\partial B_2|\|D\eta\|_{L^\infty(B_2)}^2\|\sqrt{1+|Du|^2}\|_{L^\infty(B_2)}.
   \end{split}
 \end{equation*}
Also, we have 
\begin{equation}\label{eq:graph-volume}
 \mu(M\cap\{x\in B_2\})=\int_{B_2} \sqrt{1+|Du|^2}\,dx
 \leq| B_2|\|\sqrt{1+|Du|^2}\|_{L^\infty(B_2)}.
\end{equation}
Thus, we  prove \eqref{eq:energy-fixed}.
\end{proof}

\section{The semi-convex parallel-hypersurface  transformation}\label{sec:LL}
\noindent

Let
\[
 X(x)=(x,u(x)),
 \qquad
 \nu(x)=\frac{(Du(x),-1)}{W(x)},
 \qquad
 W(x)=\sqrt{1+|Du(x)|^2},
\]
and 
\[ Y(x):=X(x)+t\nu(x)     =\left(T(x),u(x)-\frac{t}{W(x)}\right),
 \qquad y=T(x):=x+t\frac{Du(x)}{W(x)},
\]
where $t:=\frac{1}{4(K+1)}$.
We set
\begin{equation}\label{eq:P-J-definitions}
 P:=I+tA,
 \qquad \alpha_i:=1+t\kappa_i,
 \qquad J:=\det P=\prod_i\alpha_i,
\end{equation}
where $\kappa_1\geq \ldots\geq \kappa_n\geq -K$ are   the eigenvalues of $A$. The semi-convexity bound $\kappa_i\geq-K$ gives
\begin{equation}\label{eq:A-positive}
 P\geq a_0I,
 \qquad a_0:=1-tK\geq\frac34.
\end{equation}
By $d\nu=A\circ dX$, we have 
\[ dY=P\circ dX.\]
Consider parallel hypersurface $\widetilde M=\{Y(x): x\in B_{10}\}$. Then
\begin{equation}\label{eq7}
  \Phi=Y \circ X^{-1}: M\to \widetilde M
\end{equation}
and $$d\Phi=P.$$
Let \(g\) and \(\widetilde g\) be the metrics on \(M\) and
\(\widetilde M\) induced by the ambient Euclidean inner product
\(\langle\cdot,\cdot\rangle_{\mathbb R^{n+1}}\), that is,   
\[
 g(v,w)
 =\widetilde g(v,w)
 =\langle v,w \rangle_{\mathbb R^{n+1}}.
\]

\begin{lemma}[Local graph property]\label{lem:T-jacobian}
The map $T$ is a local $C^2$ diffeomorphism and
\begin{equation}\label{eq:det-horizontal}
 J(x)=\det DT(x)>0,
 \qquad x\in B_{10}.
\end{equation}
Let $U\subset B_{10}$ be any open set such that
$T|_U:U\to V:=T(U)$ is a $C^2$ diffeomorphism.
Then the restriction $Y|_U$ parametrizes the graph of a function
$\widetilde u\in C^2(V)$; that is,
\[
 Y(U)=\{(y,\widetilde u(y)):y\in V\}.
\]
Moreover,
\begin{equation}\label{eq:same-slope}
  D_y\widetilde u(y)=Du(x),
 \qquad y=T(x), \qquad x\in U.
\end{equation}
In particular, the graph factors of the original and parallel graphs
agree at corresponding points:
\[
 \sqrt{1+|D_y\widetilde u(T(x))|^2}
 =\sqrt{1+|Du(x)|^2}=W(x),
 \qquad x\in U.
\]
\end{lemma}

\begin{proof}
Let
\[
 \pi:\R^{n+1}\longrightarrow\R^n,
 \qquad
 \pi(z,z_{n+1})=z,
\]
be the horizontal projection.  Since $T=\pi\circ Y$, we have
\begin{equation}\label{eq:DT-composition}
 DT=d\pi\circ dY.
\end{equation}
We first identify the tangent space of the parallel hypersurface $\widetilde M$. 
Since $P$ is invertible, we have 
\[
\begin{aligned}
T_{Y(x)}\widetilde M
&=dY(T_xB_{10})\\
&=P\bigl(dX(T_xB_{10})\bigr)\\
&=P\bigl(T_{X(x)}M\bigr)\\
&=T_{X(x)}M.
\end{aligned}
\]
Consequently, the original and parallel tangent hyperplanes coincide and $\nu$ is the unit normal to both.

We next show that horizontal projection is an isomorphism on the parallel
tangent space.  Suppose that
\[
 Z\in T_{Y(x)}\widetilde M
 \qquad\text{and}\qquad
 \pi(Z)=0.
\]
Then $Z$ is vertical, so $Z=cE_{n+1}$ for some $c\in\R$.  Since $Z$ is
tangent to the parallel hypersurface and $\nu$ is its normal, we have 
\[
 0=\langle Z,\nu\rangle
   =c\langle E_{n+1},\nu\rangle
   =-\frac{c}{W}.
\]
From $W>0$, it follows that $c=0$.  Therefore the restriction of
$\pi$ to $T_{Y(x)}\widetilde M$ is injective.  It is a linear map between
two $n$-dimensional spaces and so is an isomorphism.  Since $dY$ is
also an isomorphism onto the parallel tangent space,
\eqref{eq:DT-composition} shows that $DT$ is invertible.  The inverse
function theorem therefore implies that $T$ is a local $C^2$
diffeomorphism.

Let $U\subset B_{10}$ be an open set such that
$T|_U:U\to V:=T(U)$ is a $C^2$ diffeomorphism and let
\[
 S:=(T|_U)^{-1}:V\longrightarrow U
\]
denote its inverse. Thus,
\[
 T\circ S=\operatorname{id}_V,
 \qquad
 S\circ(T|_U)=\operatorname{id}_U.
\]
Define $\widetilde u\in C^2(V)$ by
\begin{equation}\label{eq:utilde-definition}
 \widetilde u(y)
 :=Y^{n+1}(S(y))
 =u(S(y))-\frac{t}{W(S(y))},
 \qquad y\in V.
\end{equation}
Then, for every $y\in V$,
\[
 Y(S(y))
 =\bigl(T(S(y)),Y^{n+1}(S(y))\bigr)
 =\bigl(y,\widetilde u(y)\bigr).
\]
Since $S(V)=U$, it follows that
\[
 Y(U)=\{(y,\widetilde u(y)):y\in V\}.
\]
Hence the restriction $Y|_U$ parametrizes the graph of
$\widetilde u$ over $V$.

The downward unit normal to the graph of $\widetilde u$ over $V$ is
\[
 \widetilde\nu(y)
 =\frac{(D_y\widetilde u(y),-1)}
        {\sqrt{1+|D_y\widetilde u(y)|^2}},
 \qquad y\in V.
\]
For each $y\in V$, let $x=S(y)\in U$. By the tangent-space
identification established above, the vector
\[
 \nu(S(y))
 =\frac{(Du(S(y)),-1)}{W(S(y))}
\]
is also a unit normal to the parallel graph at
$Y(S(y))=(y,\widetilde u(y))$.  By uniqueness of the downward
unit normal at this point, we obtain
\begin{equation}\label{eq:normal-comparison}
 \frac{(D_y\widetilde u(y),-1)}
      {\sqrt{1+|D_y\widetilde u(y)|^2}}
 =
 \frac{(Du(S(y)),-1)}{W(S(y))},
 \qquad y\in V.
\end{equation}
Comparing the last components in \eqref{eq:normal-comparison} gives
\begin{equation}\label{eq:W-equality}
  \sqrt{1+|D_y\widetilde u(y)|^2}
 =W(S(y)),  \qquad y\in V.
\end{equation}
Comparing the horizontal components in
\eqref{eq:normal-comparison} and using \eqref{eq:W-equality},
we then obtain
\[
 D_y\widetilde u(y)=Du(S(y)),
 \qquad y\in V.
\]
This proves \eqref{eq:same-slope} and shows that the graph factors
of the original and parallel graphs agree at corresponding points.

We next compare two expressions for the area element of $Y(U)$.
Since \(d\Phi=P\) and \(J=\det P>0\),
\[Y^*d\tilde{\mu}=
 X^*\Phi^*d\tilde{\mu}
 =J(x)\,X^* d\mu =J(x)W(x)\,dx.
\]
On the other hand, in  $y\in V$,
the same hypersurface is parametrized by
$y\mapsto(y,\widetilde u(y))$. Hence, by \eqref{eq:W-equality}, we have 
\[ Y^*d\tilde{\mu}
 =\sqrt{1+|D_y\widetilde u(y)|^2}\,dy
 =W(S(y))\,dy.
\]
Changing variables by $y=T(x)$ with $x\in U$ and using
$S(T(x))=x$, we obtain
\[Y^* d\tilde{\mu}
 =W(x)|\det DT(x)|\,dx.\]
Comparing the two expressions in the same $x$-coordinates
and using $W(x)>0$ yields
\begin{equation}\label{eq:absolute-jacobian}
 |\det DT(x)|=J(x),
 \qquad x\in U.
\end{equation}
Since every point of $B_{10}$ belongs to such a neighborhood $U$,
this identity holds throughout $B_{10}$.

It remains to determine the sign of $\det DT$.
For $s\in[0,1]$, define
\[
 Y_s(x):=X(x)+st\nu(x),
 \qquad
 T_s(x):=x+st\frac{Du(x)}{W(x)}.\]
Then $T_s=\pi\circ Y_s$ and
\[ dY_s=P_{s}\circ dX, \qquad
 P_s:=I+stA.\]
In an orthonormal principal frame, the eigenvalues of $P_s$ are
$1+st\kappa_i$. By the semi-convexity and the choice of $t$, we see that 
\[ 1+st\kappa_i \geq 1-stK \geq 1-tK =a_0>0.
\]
Thus $P_s$ is invertible for every $s\in[0,1]$.
The same tangent-space argument used above shows that $\nu(x)$
is a unit normal to $Y_s$ at $Y_s(x)$. Since
\[ \langle\nu(x),E_{n+1}\rangle=-W(x)^{-1}\neq0,
\]
horizontal projection restricts to an isomorphism on the tangent
space of $Y_s$. Therefore,
\[ DT_s(x)=d\pi_{Y_s(x)}\circ dY_s
\]
is invertible and 
\[ \det DT_s(x)\neq0,
 \qquad x\in B_{10},\quad s\in[0,1].
\]
Fixed $x\in B_{10}$, the function
\[ s\longmapsto\det DT_s(x)
\]
is continuous and never vanishes. Moreover, from
\[ T_0=\operatorname{id},
 \qquad
 \det DT_0(x)=1,
\]
it follows that
\[
 \det DT_s(x)>0
 \qquad\text{for every }s\in[0,1].
\]
Taking $s=1$ and using \eqref{eq:absolute-jacobian}, we conclude that
\[
 \det DT(x)=J(x)>0,
 \qquad x\in B_{10}.
\]
This proves \eqref{eq:det-horizontal}.
\end{proof}

\begin{lemma}[Global graph representation]
\label{lem:ball-inclusion}
Let $B_R(x_c)\Subset B_{10}$ with $R>t$ and set
\[
 V:=B_{R-t}(x_c),
 \qquad
 U:=B_R(x_c)\cap T^{-1}(V).
\]
Then $T|_U:U\to V$ is a $C^2$ diffeomorphism.
Consequently, $Y(U)$ is the graph of a function
$\widetilde u\in C^2(V)$ over $V$.
\end{lemma}

\begin{proof}
We divide the proof into three steps.

\medskip
\noindent
\textbf{Step 1: existence of a preimage for every point of $V$.}
For $s\in[0,1]$, define
\[T_s(x):=x+st\frac{Du(x)}{W(x)}.
\]
Then $ T_0=\operatorname{id}$ and $ T_1=T$.
Since
\[ \frac{|Du(x)|}{W(x)}
 =\frac{|Du(x)|}{\sqrt{1+|Du(x)|^2}}\leq1,\]
we have
\[ |T_s(x)-x|\leq st\leq t.\]
Fix $y\in V$. For every $x\in\partial B_R(x_c)$ and
$s\in[0,1]$, the triangle inequality gives
\begin{align*}
 | T_s(x)-y|&\geq |x-x_c|-| T_s(x)-x|-|y-x_c|\\
 &\geq R-st-|y-x_c|\\
 &\geq R-t-|y-x_c|>0.
\end{align*}
Consequently, we have 
\[
 y\notin T_s(\partial B_R(x_c))
 \qquad\text{for every }s\in[0,1].
\]
In view of  $\overline{B_R(x_c)}\subset B_{10}$, the homotopy
$ T_s$ is continuous on
$[0,1]\times\overline{B_R(x_c)}$.
The homotopy invariance of the Brouwer degree  yields
\begin{align*}
 \deg(T,B_R(x_c),y)
=\deg(\operatorname{id},B_R(x_c),y)=1,
\end{align*}
where the last equality follows from $y\in V\subset B_R(x_c)$.
Since the degree is nonzero, the equation $T(x)=y$ has at least
one solution in $B_R(x_c)$. This proves
\begin{equation}\label{eq:ball-inclusion}
 V=B_{R-t}(x_c)\subset T(B_R(x_c)).
\end{equation}

\medskip
\noindent
\textbf{Step 2: uniqueness of the preimage.}
Fix $y\in V$ and consider 
\[ E_y:=\{x\in\overline{B_R(x_c)}:T(x)=y\}.
\]
This set is compact by continuity of $T$. The boundary exclusion
proved in Step~1 shows that
\[ E_y\subset B_R(x_c).\]
Moreover, Lemma~\ref{lem:T-jacobian} gives
\[ \det DT(x)=J(x)>0
 \qquad\text{for every }x\in B_R(x_c).\]
Since $T$ is locally injective at every point of $E_y$,  every
point of $E_y$ is isolated. Then compactness of $E_y$ implies that $E_y$
is finite.
The formula for the Brouwer degree at a regular value now gives
\[ 1 =\deg(T,B_R(x_c),y) =\sum_{x\in E_y}\operatorname{sgn}\det DT(x)
 =\#E_y.\]Hence $y$ has exactly one preimage in $B_R(x_c)$.

By the definition of $U$, this unique preimage belongs to $U$.
It follows that
\[
 T|_U:U\longrightarrow V
\]
is both injective and surjective.

\medskip
\noindent
\textbf{Step 3: regularity of the inverse.}
The set $U$ is open, since $B_R(x_c)$ and $V$ are open and $T$
is continuous. Define the inverse of the bijection $T|_U$ by
\begin{equation}\label{eq:inverse-branch}
 S_{x_c,R}:=(T|_U)^{-1}:V\longrightarrow U.
\end{equation}
Then
\[
 T\circ S_{x_c,R}=\operatorname{id}_V,
 \qquad
 S_{x_c,R}\circ(T|_U)=\operatorname{id}_U.
\]

To prove that $S_{x_c,R}$ is $C^2$, fix $y_0\in V$ and let
$x_0=S_{x_c,R}(y_0)$.
Since $DT(x_0)$ is invertible, the inverse function theorem
provides open neighborhoods $U_0\subset U$ of $x_0$ and
$V_0\subset V$ of $y_0$ such that
\[
 T|_{U_0}:U_0\longrightarrow V_0
\]
is a $C^2$ diffeomorphism.
For every $y\in V_0$, both $(T|_{U_0})^{-1}(y)$ and
$S_{x_c,R}(y)$ belong to $B_R(x_c)$ and are mapped to $y$ by $T$.
The uniqueness established in Step~2 therefore implies
\[
 S_{x_c,R}|_{V_0}=(T|_{U_0})^{-1}.
\]
Thus $S_{x_c,R}$ is $C^2$ near every point of $V$ and 
$S_{x_c,R}\in C^2(V;U)$.
This proves that $T|_U$ is a $C^2$ diffeomorphism.

The graph representation of $Y(U)$ now follows from
Lemma~\ref{lem:T-jacobian}.
\end{proof}


Let
\[
 \Omega:=X(U)\subset M,\qquad
 \widetilde\Omega:=Y(U)\subset\widetilde M,
\]
where \(U\) is  selected in Lemma~\ref{lem:ball-inclusion}.
For a function \(v\) on \(\Omega\), define its corresponding function
\(v^*\) on \(\widetilde\Omega\) by
\begin{equation}\label{eq:surface-function-change}
 v^*\circ\Phi=v.
\end{equation}
This implies that 
\begin{equation}\label{eq:surface-gradient-change}
 \nabla v
 =P\bigl(\nabla v^*\bigr),
 \qquad
 \nabla v^*
 =P^{-1}\nabla v.
\end{equation}
Indeed, for every \(V\in T_pM\),
\[
 \begin{aligned}
 \langle\nabla v,V\rangle
 &=dv (V)
 =d v^*(d\Phi V)\\
 &= \langle\nabla v^*,PV\rangle\\
 &=\langle P \nabla v^*,V\rangle,
 \end{aligned}
\]
where the last equality uses the self-adjointness of \(P\).

For each fixed point \(p\in\Omega\), we set 
\[ F:=(F_{ij})=H\delta_{ij}-h_{ij}\] and 
\begin{equation}\label{eq:B-def}
\mathcal B  :=\frac1{J}PFP.
\end{equation}
Then \begin{equation}\label{eq:B-bilinear-form}
\langle \mathcal B v,w\rangle
 =\frac1{J}\langle FPv,Pw\rangle , \qquad v,w\in T_pM.
\end{equation}

\begin{proposition}[Transformation rule]
Let $n\geq 2$, $K\geq 0$, $u\in C^4(B_{10})$ and $f\in C^2(B_{10})$ with $D^2u\geq -KI$ and $\inf_{B_{10}}f >0$. If $M=\{(x,u(x)):x\in B_{10}\}\subset\R^{n+1}$ is a smooth 2-convex graph satisfying the scalar curvature equation \eqref{eq:intro-main}, then there exists some constant  $C>0$ depending only on $n$, $K$, $\inf_{B_{10}}f$, $\sup_{B_{10}}f$ and $\|Df\|_{L^{\infty}(B_{10})}$ such that
\begin{equation}\label{eq:pushed-exact}
 \operatorname{div}_{\widetilde{M}}
   (\mathcal B^* \nabla b^*)
 \geq  \operatorname{div} \mathcal Q
 -C\left(\frac HJ\circ\Phi^{-1}\right)\ \ \mbox{in}\ \ \ \widetilde\Omega
\end{equation}
in the sense of distributions, where 
\begin{equation}\label{eq:q-and-pushed-fields}
 b=\log(H+J_0),\qquad
 q:=\frac{\nabla f}{H+J_0}
 \quad\text{in }\ \Omega,
\end{equation}
and  
\begin{equation}\label{eq:push-fields}
 \mathcal B^{*}=\mathcal B \circ\Phi^{-1}, \qquad b^*:=b\circ\Phi^{-1},\qquad
 \mathcal Q : =\frac{Pq}{J}\circ\Phi^{-1}\quad \text{in}\ \ \widetilde\Omega.
\end{equation}
Furthermore, in the horizontal coordinates
\(\widetilde X(y)=(y,\widetilde u(y))\), for 
\(\beta=b^*\circ\widetilde X\),  then we have 
\begin{equation}\label{eq:pushed-euclidean-form}
 \partial_a\!\left(A^{ab}\partial_b\beta\right)
\geq
\partial_aG^a
-C\widetilde W
\left[
\left(\frac{H}{J}\right)
\circ\Phi^{-1}\circ\widetilde X
\right]
\qquad\text{in }V
\end{equation}
where $$A^{ab}
:=\widetilde W
\left\langle
\mathcal B^*\nabla_{\widetilde M}y^b,
\nabla_{\widetilde M}y^a
\right\rangle,
\qquad
G^a
:=\widetilde W\,dy^a(\mathcal Q).$$
\end{proposition}

\begin{proof}
In the weak Jacobi inequality \eqref{eq:weak-test}, the quadratic term
in \(\nabla b\) is nonnegative.  After dropping this term, every
nonnegative function \(\phi\in C_c^1(\Omega)\) satisfies
\begin{equation}\label{eq:weak-dropped}
 -\int_\Omega \langle F \nabla b,\nabla\phi\rangle\,d\mu
 \geq -\int_\Omega \langle q,\nabla\phi\rangle\,d\mu
 -C\int_\Omega H\phi\,d\mu
\end{equation}
for some constant  $C>0$ depending only on $n$, $K$, $\inf_{B_{10}}f$, $\sup_{B_{10}}f$ and $\|Df\|_{L^{\infty}(B_{10})}$ .
For any nonnegative function  \(0\leq\psi\in C_c^1(\widetilde\Omega)\), we set
\[ \phi:=\psi\circ\Phi\in C_c^1(\Omega).
\]
Then the corresponding function \(\phi^*=\phi\circ\Phi^{-1}\) is
exactly \(\psi\).  
The gradient transformation gives
\[
 \nabla b=P \nabla b^*,
 \qquad
 \nabla\phi=P \nabla\phi^*.
\]
Combining \eqref{eq:B-bilinear-form}  with
\[
 \Phi^*d\widetilde{\mu}=J\,d\mu, 
\]
we prove \begin{equation}\label{eq:energy-change}
 \int_\Omega \langle F \nabla b,\nabla\phi\rangle\,d\mu
 = \int_{\widetilde\Omega} \bigl\langle\mathcal B^{*} \nabla b^*,
       \nabla\phi^*\bigr\rangle \,d\widetilde{\mu}.
\end{equation}
By \[
 \begin{aligned}
 \langle\mathcal Q, \nabla\phi^*\rangle
 &= \frac1J\, (Pq, \nabla\phi^*)  \\
 &= \frac1J\,d\phi^*(Pq)
 =\frac1J\,d\phi(q)
 =\frac1J \langle q,\nabla\phi \rangle,
 \end{aligned}
\]
we have 
\begin{equation}\label{eq:flux-change}
 \int_\Omega \langle q,\nabla\phi\rangle \,d\mu
 =  \int_{\widetilde\Omega}
\langle\mathcal Q, \nabla\phi^*\rangle
 \,d\widetilde{\mu}.
\end{equation}
For the last term, the area change formula gives
\[\int_\Omega H(\psi\circ\Phi)\,d\mu
 =\int_{\widetilde\Omega}
 \left(\frac HJ\circ\Phi^{-1}\right)
 \psi\,d\widetilde{\mu}.
\]

We now express the transformed weak inequality in the horizontal
coordinates of the parallel graph.  Let
\[
\widetilde X(y)=(y,\widetilde u(y)),\qquad y\in V,
\]
and set
\[
\widetilde g_{ab}
=\delta_{ab}+\widetilde u_a\widetilde u_b,
\qquad
\widetilde W
=\sqrt{1+|D\widetilde u|^2}.
\]
For
\[
\beta:=b^*\circ\widetilde X,
\qquad
\varphi:=\psi\circ\widetilde X,
\]
we have
\[
\nabla_{\widetilde M}b^*
   =\beta_b\nabla_{\widetilde M}y^b,
\qquad
\nabla_{\widetilde M}\psi
   =\varphi_a\nabla_{\widetilde M}y^a,
\qquad
d\widetilde\mu=\widetilde W\,dy.
\]
Define
\begin{equation}\label{eq:euclidean-coefficients}
A^{ab}
:=\widetilde W
\left\langle
\mathcal B^*\nabla_{\widetilde M}y^b,
\nabla_{\widetilde M}y^a
\right\rangle,
\qquad
G^a
:=\widetilde W\,dy^a(\mathcal Q).
\end{equation}
Then
\[
\int_{\widetilde\Omega}
\left\langle
\mathcal B^*\nabla_{\widetilde M}b^*,
\nabla_{\widetilde M}\psi
\right\rangle\,d\widetilde\mu
=
\int_V A^{ab}\beta_b\varphi_a\,dy,
\]
and
\[
\int_{\widetilde\Omega}
\left\langle\mathcal Q,
\nabla_{\widetilde M}\psi
\right\rangle\,d\widetilde\mu
=
\int_V G^a\varphi_a\,dy.
\]
Consequently, for every nonnegative
$\varphi\in C_c^1(V)$,
\[
-\int_V A^{ab}\beta_b\varphi_a\,dy
\geq
-\int_V G^a\varphi_a\,dy
-C\int_V
\widetilde W
\left[
\left(\frac{H}{J}\right)
\circ\Phi^{-1}\circ\widetilde X
\right]\varphi\,dy.
\]
Equivalently,
\begin{equation}\label{eq:pushed-euclidean-form}
\partial_a\!\left(A^{ab}\partial_b\beta\right)
\geq
\partial_aG^a
-C\widetilde W
\left[
\left(\frac{H}{J}\right)
\circ\Phi^{-1}\circ\widetilde X
\right]
\qquad\text{in }V
\end{equation}
in the sense of distributions.

\end{proof}

\begin{lemma}[Uniform  ellipticity]
\label{lem:uniform-B}
There exists some constant  $C>0$ depending only on $n$, $K$, $\inf_{B_{10}}f$, $\sup_{B_{10}}f$ and $\|Df\|_{L^{\infty}(B_{10})}$ such that
\begin{equation}\label{eq:B-uniform}
 \frac{1}{C} I\leq\mathcal A^{ab} \leq C I,\qquad \|G^a\|_{L^\infty(\widetilde \Omega)}  \leq a_0^{-(n-1)}\|Df\|_{L^{\infty}(B_{10})}
\end{equation}
and 
\begin{equation}\label{eq:J-bounds}
\frac{1}{C}(1+H)\le J\le C(1+H).
\end{equation}
\end{lemma}

\begin{proof}

For fixed \(p\in\Omega\), we set \(\widetilde p=\Phi(p)\) and choose an
orthonormal principal frame \(\{e_i\}\)  such that 
\[ P e_i=\alpha_i e_i,\qquad
 F_{ij}=(H-\kappa_i)\delta_{ij}.
\]
Substituting these identities into \eqref{eq:B-def} gives
$$\mathcal B_i =\frac{\alpha_i^2(H-\kappa_i)}{J}=\frac{(1+t\kappa_i)^2(H-\kappa_i)}
{\prod_j(1+t\kappa_j)}.$$
It remains to obtain uniform upper and
lower bounds for these eigenvalues.

Choose \(S_0\geq1\) sufficiently large such that
Lemma~\ref{lem:large-trace-spectrum} can be applied to the case  \(H\geq S_0\).
We divide the argument into two cases.

\smallskip
\noindent
\textbf{Case 1: \(H\geq S_0\).}
 Lemma~\ref{lem:large-trace-spectrum} gives
\[
 \frac{H}{n}\leq\kappa_1\leq H,\qquad
 \frac{1}{CH}\leq F_{11}\leq\frac CH,
\]
and,
\[
 |\kappa_i|\leq C,\qquad
 \frac{H}{C}\leq F_{ii}\leq CH,\ \ \ i\geq 2
\]
for some constant $C>0$.
Since \(t=1/[4(K+1)]\) and
\(\alpha_i=1+t\kappa_i\geq a_0>0\), these estimates imply
\[
 \frac{1}{C}H\leq\alpha_1\leq CH,\qquad
 a_0\leq\alpha_i\leq C\quad(i\geq2).
\]
Therefore, we have 
\begin{equation}\label{eq:J-large-proof}
 \frac{H}{C}\leq J=\alpha_1\prod_{i=2}^n\alpha_i\leq CH.
\end{equation}
Moreover, we see that 
\[
 \frac{H}{C}\leq\alpha_1^2F_{11}\leq CH,
\]
and 
\[
\frac{H}{C}\leq\alpha_i^2F_{ii}\leq CH\ \ \, i\geq 2.
\]
Division by \eqref{eq:J-large-proof} yields
\[
 \frac{1}{C}\leq\mathcal B_i\leq C
 \qquad i\geq 1.
\]

\smallskip
\noindent
\textbf{Case 2: \(0<H\leq S_0\).}
By Lemma \ref{lem1}, we have 
\[
 F_{ii}\geq\frac fH\geq\frac{\inf_{B_{10}}f}{S_0}.
\]
The semi-convexity bound \(\kappa_i\geq-K\) also gives
\[
 F_{ii}=H-\kappa_i\leq S_0+K.
\]
Moreover, \(|\kappa|<H\leq S_0\) and \(\alpha_i\geq a_0\) imply
\[
 a_0\leq\alpha_i\leq1+tS_0,
 \qquad
 a_0^n\leq J\leq(1+tS_0)^n.
\]
Consequently,
\[
 \frac{a_0^2f_0}{S_0(1+tS_0)^n}
 \leq\mathcal B_i
 \leq
 \frac{(1+tS_0)^2(S_0+K)}{a_0^n}.
\]
Combining the two cases proves \eqref{eq:B-uniform} and \eqref{eq:J-bounds}.

Set
\[
\vartheta_i^a:=dy^a(e_i),
\qquad
\mathcal Q_i:=\langle\mathcal Q,e_i\rangle.
\]
Since
\[
\nabla_{\widetilde M}y^a
=\sum_i\vartheta_i^a e_i,
\]
the coefficients in \eqref{eq:euclidean-coefficients} become
\[
A^{ab}
=\widetilde W\sum_i
\mu_i\vartheta_i^a\vartheta_i^b,
\qquad
G^a
=\widetilde W\sum_i
\mathcal Q_i\vartheta_i^a.
\]
Therefore, for every $\xi\in\mathbb R^n$,
\[
A^{ab}\xi_a\xi_b
=
\widetilde W\sum_i\mu_i
\left(\sum_a\vartheta_i^a\xi_a\right)^2.
\]
Moreover,
\[
\sum_i
\left(\sum_a\vartheta_i^a\xi_a\right)^2
=
\widetilde g^{ab}\xi_a\xi_b.
\]
It follows that
\[
c_0\widetilde W\,\widetilde g^{ab}\xi_a\xi_b
\leq
A^{ab}\xi_a\xi_b
\leq
C_0\widetilde W\,\widetilde g^{ab}\xi_a\xi_b.
\]
Since
\[
\widetilde g^{ab}
=\delta_{ab}
-\frac{\widetilde u_a\widetilde u_b}{\widetilde W^2},
\]
we obtain
\[
\frac{c_0}{\widetilde W}|\xi|^2
\leq
A^{ab}\xi_a\xi_b
\leq
C_0\widetilde W|\xi|^2.
\]
Finally, the slope bound gives
\[
1\leq\widetilde W
\leq
W_0:=\sqrt{1+\|Du\|_{L^\infty(B_{10})}^2},
\]
and hence
\[
\frac{c_0}{W_0}|\xi|^2
\leq
A^{ab}\xi_a\xi_b
\leq
C_0W_0|\xi|^2.
\]
Thus $(A^{ab})$ is uniformly elliptic in the Euclidean coordinates.

Furthermore,
\[
|G|
=\widetilde W\,|d\pi(\mathcal Q)|
\leq
\widetilde W|\mathcal Q|
\leq
W_0\|\mathcal Q\|_{L^\infty(\widetilde\Omega)},
\]
while $\widetilde W H/J$ is uniformly bounded.  
We next estimate \(\mathcal Q\). 
By \(H+J_0\geq1\), it follows that
\[
 |q|
 =\frac{|\nabla f|}{H+J_0}
 \leq \|Df\|_{L^{\infty}(B_{10})}.
\]
Also, combining this  with 
\[
 \frac{\alpha_i}{J}
 =\frac1{\prod_{j\ne i}\alpha_j}
 \leq a_0^{-(n-1)},
\]
we see that 
\[
 \|\mathcal Q\|_{L^\infty(\widetilde \Omega)}
 \leq a_0^{-(n-1)}\|Df\|_{L^{\infty}(B_{10})}.
\]
\end{proof}

We introduce the following  local boundedness estimate from \cite[Theorem 8.17]{GT83} and \cite{HL11}.

\begin{lemma}\label{lem:DG}
Let \(G\in L^\infty(B_1;\mathbb R^n)\), \(E\in L^\infty(B_1)\) and
let \(\mathbb A=(a^{ij})\) satisfy
\[
\lambda_0 I\leq\mathbb A\leq\Lambda_0I.
\]
Suppose that \(v\in W^{1,2}(B_1)\) satisfies
\[
\operatorname{div}(\mathbb A Dv)
\geq \operatorname{div}G-E
\qquad\text{in }B_1.
\]
Then
\begin{equation*}
 \sup_{B_{1/2}}u
 \le C\left(1+\|v_+\|_{L^1(B_1)}
 +\|G\|_{L^\infty(B_1)}+\|E\|_{L^\infty(B_1)}\right),
\end{equation*}
where $C>0$ depends only on $n$, $\lambda_0$ and $\Lambda_0$.
\end{lemma}
\bigskip

Since \(J_0>1\) and \(H>0\), we have \(b=\log(H+J_0)\ge0\).
Fix \(x_0\in B_{1/2}\) and set \(y_0=T(x_0)\). Since
\(|y_0-x_0|\le t \le1/4\), if $|y-y_0|<\frac18$, then 
\[ |y-x_0|<\frac18+t\le\frac34-t.\]
Thus, applying Lemma~\ref{lem:ball-inclusion} to \(B_{3/4}(x_0)\) yields
\begin{equation}\label{eq:fixed-transformed-ball} 
 B_{1/8}(y_0)\subset B_{3/4-t}(x_0)
 \subset T(B_{3/4}(x_0)).
\end{equation}
Let \(S=S_{x_0,3/4}\) be the inverse branch supplied by
Lemma~\ref{lem:ball-inclusion} and by $\widetilde X(y)=\Phi(X(S(y)))$,  we have 
\[
 \beta(y)=b^*(\widetilde X(y))=b(X(S(y))).
\]
Applying Lemma~\ref{lem:DG} to
\eqref{eq:pushed-euclidean-form} gives
\begin{equation}\label{eq:mean-value-y}
 b(X(x_0))=\beta(y_0)
 \le C\left(1+
       \int_{B_{1/8}(y_0)}\beta(y)\,dy\right).
\end{equation}
Finally, \(\det DT=J>0\), \(d\mu=Wdx\), and
\eqref{eq:J-bounds} imply
\begin{align}
 \int_{B_{1/8}(y_0)}\beta(y)\,dy
 &=\int_{S(B_{1/8}(y_0))}b(X(x))J(x)\,dx\notag\\
 &={
 \int_{X(S(B_{1/8}(y_0)))}
       b\,\frac{J}{W}\,d\mu_g}\notag\\
 &\le C\int_{M\cap\{x\in B_{5/4}\}}
       b(1+H)\,d\mu.                         \label{eq:mean-change}
\end{align}
Consequently,
\begin{equation}\label{eq:mean-reduction}
 \sup_{x_0\in B_{1/2}}b(X(x_0))
 \le C\left(1+
 \int_{M\cap\{x\in B_{5/4}\}}b(1+H)\,d\mu_g\right).
\end{equation}

\section{Weighted integral closure and proof of the theorem}
\noindent

It remains to control the weighted integral in
\eqref{eq:mean-reduction}.

\begin{lemma}\label{lem:closure}
Let $n\geq 2$, $K\geq 0$, $u\in C^4(B_{10})$ and $f\in C^2(B_{10})$ with $D^2u\geq -KI$ and $\inf_{B_{10}}f >0$. If $M=\{(x,u(x)):x\in B_{10}\}\subset\R^{n+1}$ is a smooth 2-convex graph satisfying the scalar curvature equation \eqref{eq:intro-main}, then there exists some constant  $C>0$ depending only on $n$, $K$, $\inf_{B_{10}}f$, $\sup_{B_{10}}f$, $\|Df\|_{L^{\infty}(B_{10})}$ and $\|Du\|_{L^{\infty}(B_{10})}$ such that
\begin{equation}\label{eq:closure}
 \int_{M\cap\{x\in B_{5/4}\}}b(1+H)\,d\mu_g
 \le C.
\end{equation}
\end{lemma}

\begin{proof}
For the downward normal \(\nu=(Du,-1)/W\) and the convention
\(d\nu=A\), we know that
\[
 H=\operatorname{div} \!\left(\frac{Du}{W}\right),
 \qquad d\mu =W\,dx.
\]
Since $b=\log(H+J_0)$ and $J_0\geq1$,
\begin{equation}\label{eq:log-linear}
                         0\leq b\leq H+J_0.
\end{equation}
Combining this with  \eqref{eq:L1-s}   gives
\begin{equation}\label{eq:b-L1}
 \int_{M\cap\{x\in B_{3/2}\}}b\,d\mu \leq C,
\end{equation}
where $C>0$ depends only on $n$, $K$, $\inf_{B_{10}}f$, $\sup_{B_{10}}f$ and $\|Du\|_{L^{\infty}(B_2)}$.

It remains to estimate the term containing \(bH\). Choose
\(\chi\in C_c^1(B_{3/2})\) with \(0\le\chi\le1\) and
\(\chi=1\) on \(B_{5/4}\). Since \(bH\ge0\), we get 
\begin{align}
 \int_{M\cap\{x\in B_{5/4}\}}bH\,d\mu
 &\le \|W\|_{L^\infty(B_{3/2})}
       \int_{B_{3/2}}\chi bH\,dx\notag\\
 &\le \|W\|_{L^\infty(B_{3/2})}\left(\left|
       \int_{B_{3/2}}\chi D b\cdot\frac{Du}{W}\,dx\right|
       +\int_{B_{3/2}}b|D\chi|\,dx\right).          \label{eq:closure-IBP}
\end{align}
In the graph coordinates, we have 
\[\zeta=\frac{1}{W^2}\sum_{i=1}^n u_iX_i,
\qquad
|\zeta|\leq1,
\qquad
\langle\nabla_M b,\zeta\rangle
=\frac{D\bar b\cdot Du}{W^2},
\]
where \(\bar b:=b\circ X\).
Since \(d\mu=W\,dx\), we have
\[
\int_{B_{3/2}}\chi D\bar b\cdot\frac{Du}{W}\,dx
=
\int_{M_{3/2}}\chi
\langle\nabla_M b,\zeta\rangle\,d\mu.
\]
Therefore, the weighted Cauchy--Schwarz inequality gives
\begin{align*}
\left|
\int_{B_{3/2}}\chi D\bar b\cdot\frac{Du}{W}\,dx
\right|
&\leq
\left(
\int_{M\cap \{x\in B_{3/2}\}}\chi^2
\langle F\nabla_Mb,\nabla_Mb\rangle\,d\mu
\right)^{1/2} \\
&\quad\times
\left(
\int_{M\cap \{x\in B_{3/2}\}}
\langle F^{-1}\zeta,\zeta\rangle\,d\mu
\right)^{1/2}.
\end{align*}
The first factor is bounded by Proposition~\ref{prop:energy} applied
with \(\eta=\chi\).
Moreover, Lemma~\ref{lem1} implies
\[
F^{-1}\leq\frac{H}{\inf_{B_{10}}f}\,I.
\]
Since \(|\zeta|\leq1\), it follows that
\[\int_{M\cap \{x\in B_{3/2}\}}\langle F^{-1}\zeta,\zeta\rangle \,d\mu
\leq \frac{1}{\inf_{B_{10}}f} \int_{M\cap \{x\in B_{3/2}\}}H\,d\mu
\leq C.
\]
Consequently,
\[\left|\int_{B_{3/2}}\chi D\bar b\cdot\frac{Du}{W}\,dx\right|
\leq C
\]
for some constant $C>0$ depending only on $n$, $K$, $\inf_{B_{10}}f$, $\sup_{B_{10}}f$, $\|Df\|_{L^{\infty}(B_{10})}$ and $\|Du\|_{L^{\infty}(B_{10})}$. Substitution into
\eqref{eq:closure-IBP} gives
\[
 \int_{M\cap\{x\in B_{5/4}\}}bH\,d\mu \le C.
\]
Together with \eqref{eq:b-L1}, this proves \eqref{eq:closure}.
\end{proof}

\begin{proof}[Proof of Theorem~\ref{thm:main}]
Combining \eqref{eq:mean-reduction} with Lemma~\ref{lem:closure} gives
\[
 \sup_{x\in B_{1/2}}b(X(x))\le C,
\]
where $C>0$ depends only on $n$, $K$, $\inf_{B_{10}}f$, $\sup_{B_{10}}f$, $\|Df\|_{L^{\infty}(B_{10})}$ and $\|Du\|_{L^{\infty}(B_{10})}$
Since \(b=\log(H+J_0)\),
\[
 \sup_{x\in B_{1/2}}H(x)\le e^C.
\]
In view of the identity \(|A|^2=H^2-2f\) and \(f>0\), we conclude that 
\[
 \sup_{B_{1/2}}|A|\leq \sup_{B_{1/2}} H\le C.
\]
\end{proof}

\section*{Acknowledgments}
 The authors acknowledge the use
of AI tools. All mathematical arguments and proofs in the final manuscript were independently verified and written by the authors.


\begin{thebibliography}{99}

\bibitem{ChenJianTuZhou}
R.~Chen, H.~Jian, X.~Tu, and X.~Zhou,
\emph{Regularity for convex viscosity solutions of the $\sigma_2$ equation},
arXiv:2605.30823 (2026).

\bibitem{ChenJianZhou}
R.~Chen, H.~Jian, and X.~Zhou,
\emph{An integral approach to prescribing scalar curvature equations},
Math. Ann. \textbf{394} (2026), Article No.~72.

\bibitem{CZZ26}
R. Chen, X. Zhou, and R. Zhu,
\emph{Interior $C^{2,\alpha}$ regularity for the quadratic Hessian equation},
arXiv preprint arXiv:2608.29484, 2026.

\bibitem{Fan}
Z.~Fan,
\emph{Hessian estimates for the $\sigma_2$ equation with variable right-hand side terms in dimension four},
Adv. Math. \textbf{494} (2026), Article No.~110953.

\bibitem{FanShankar}
Z. Fan and R. Shankar,
\emph{Interior estimates and regularity for the scalar curvature equation
in dimension four},
arXiv:2608.22748 (2026).

\bibitem{GT83}
D. Gilbarg and N.~S. Trudinger,
\emph{Elliptic Partial Differential Equations of Second Order},
second ed., Grundlehren Math. Wiss. 224, Springer, Berlin, 1983.

\bibitem{GuanQiu}
P.~Guan and G.~Qiu,
\emph{Interior $C^2$ regularity of convex solutions to prescribing scalar curvature equations},
Duke Math. J. \textbf{168} (2019), 1641--1663.

\bibitem{HL11}
Q. Han and F. Lin,
\emph{Elliptic Partial Differential Equations},
second ed., Courant Lecture Notes in Mathematics 1,
American Mathematical Society, Providence, RI, 2011.

\bibitem{Heinz}
E.~Heinz,
\emph{On elliptic Monge--Amp\`ere equations and Weyl's embedding problem},
J. Analyse Math. \textbf{7} (1959), 1--52.
\bibitem{QiuYan}
G.~Qiu and J.~Yan,
\newblock Interior curvature estimates for the graphical scalar curvature
equation in all dimensions,
\newblock arXiv:2609.02581, 2026.

\bibitem{JiLiang26}
K. Ji and L. Liang,
\emph{Interior Hessian estimates for semiconvex solutions of the
$\sigma_2/\sigma_1$ equation with Lipschitz right-hand sides},
arXiv preprint arXiv:2608.24530, 2026.

\bibitem{LiWuQuadratic}
Z.~Li and K.~Wu,
\emph{Interior Hessian estimates for the quadratic Hessian equation},
arXiv:2608.23233 preprint (2026).

\bibitem{McGonagleSongYuan}
M.~McGonagle, C.~Song, and Y.~Yuan,
\emph{Hessian estimates for convex solutions to quadratic Hessian equation},
Ann. Inst. H. Poincar\'e C Anal. Non Lin\'eaire \textbf{36} (2019), 451--454.

\bibitem{Mooney}
C.~Mooney,
\emph{Strict $2$-convexity of convex solutions to the quadratic Hessian equation},
Proc. Amer. Math. Soc. \textbf{149} (2021), 2473--2477.

\bibitem{Nir53}
L. Nirenberg,
\emph{The Weyl and Minkowski problems in differential geometry in the large},
Comm. Pure Appl. Math. \textbf{6} (1953), 337--394.
MR0058265.

\bibitem{Pog73}
A. V. Pogorelov,
\emph{Extrinsic Geometry of Convex Surfaces},
Translations of Mathematical Monographs, Vol.~35,
American Mathematical Society, Providence, RI, 1973.
Translated from the Russian by the Israel Program for Scientific
Translations.
MR0346714.

\bibitem{Pog78}
A. V. Pogorelov,
\emph{The Minkowski Multidimensional Problem},
Scripta Series in Mathematics,
V.~H. Winston \& Sons, Washington, DC;
Halsted Press (John Wiley \& Sons), New York--Toronto--London, 1978.
Translated from the Russian by Vladimir Oliker,
with an introduction by Louis Nirenberg.
MR0478079.

\bibitem{QiuCurvature}
G.~Qiu,
\emph{Interior curvature estimates for hypersurfaces of prescribing scalar curvature in dimension three},
Amer. J. Math. \textbf{146} (2024), 579--605.

\bibitem{QiuHessian}
G.~Qiu,
\emph{Interior Hessian estimates for $\sigma_2$ equations in dimension three},
Front. Math. \textbf{19} (2024), 577--598.

\bibitem{ShankarYuanSemiconvex}
R.~Shankar and Y.~Yuan,
\emph{Hessian estimate for semiconvex solutions to the $\sigma_2$ equation},
Calc. Var. Partial Differential Equations \textbf{59} (2020), Paper No.~30.

\bibitem{ShankarYuanFour}
R.~Shankar and Y.~Yuan,
\emph{Hessian estimates for the $\sigma_2$ equation in dimension four},
Ann. of Math. (2) \textbf{201} (2025), 489--513.

\bibitem{Urbas}
J.~I. E. Urbas,
\emph{On the existence of nonclassical solutions for two classes of fully nonlinear elliptic equations},
Indiana Univ. Math. J. \textbf{39} (1990), 355--382.

\bibitem{WarrenYuan}
M.~Warren and Y.~Yuan,
\emph{Hessian estimates for the $\sigma_2$ equation in dimension three},
Comm. Pure Appl. Math. \textbf{62} (2009), 305--321.

\bibitem{Zhou}
X.~Zhou,
\emph{Notes on generalized special Lagrangian equation},
Calc. Var. Partial Differential Equations \textbf{63} (2024), Paper No.~197.

\bibitem{ZZ26}
X. Zhou and R. Zhu,
\emph{Interior $C^{2,\alpha}$ regularity for convex solutions of the
$2$-Hessian equation},
arXiv preprint arXiv:2608.24604, 2026.

\end{thebibliography}
\end{document}